\documentclass[bj]{imsart}
\usepackage{amsfonts,amsmath,amssymb,amsthm,mathtools,commath,bbm,bm}
\usepackage[svgnames,x11names]{xcolor}
\usepackage{hyperref}
\usepackage[numbers]{natbib}
\usepackage{enumitem}
\usepackage{graphicx,caption,subcaption}
\usepackage[capitalize,noabbrev]{cleveref}
\usepackage[toc,page]{appendix}
\usepackage{multibib} 
\newcites{S}{References}

\newtheorem{theorem}{Theorem}[section]
\newtheorem{lemma}[theorem]{Lemma}
\newtheorem{proposition}[theorem]{Proposition}
\newtheorem{corollary}[theorem]{Corollary}
\theoremstyle{definition}
\newtheorem{definition}{Definition}
\newtheorem{example}{Example}

\newtheorem{assumption}{Assumption}
\newenvironment{assumption*}[1]{\renewcommand\theassumption{#1}\assumption}{\endassumption}
\newtheorem{remark}{Remark}

\numberwithin{equation}{section}
\newenvironment{equations}{\equation\aligned}{\endaligned\endequation}

\hypersetup{colorlinks=True, citecolor=Blue, linkcolor=Navy, urlcolor=Blue}

\makeatletter
\newcommand*{\rom}[1]{\expandafter\@slowromancap\romannumeral #1@}
\makeatother

\DeclareMathOperator*{\argmin}{\arg\!\min}

\newcommand{\floor}[1]{\left\lfloor {#1} \right\rfloor}
\newcommand{\ceil}[1]{\left\lceil {#1} \right\rceil}

\def\iidsim{\stackrel{\textup{iid}}{\sim}}
\def\indsim{\stackrel{\textup{ind}}{\sim}}

\def\ind{\mathbbm{1}}

\def\var{\mathsf{Var}}
\def\kl{\mathsf{KL}}

\def\BINOM{\textsc{binom}}
\def\GAMMA{\textsc{gamma}}
\def\BETA{\textsc{beta}}
\def\Pois{\textsc{poisson}}

\def\unif{\textsc{unif}}
\def\Dir{\textsc{dir}}
\def\DP{\textsc{dp}}

\def\d{\textup{d}}
\def\e{\textup{e}}

\def\N{\textsc{N}}

\def\E{\mathsf{E}}
\def\P{\mathsf{P}}
\def\R{\mathbb{R}}

\makeatletter
\def\greekvectors#1{%
 \@for\next:=#1\do{%
    \def\X##1;{\expandafter\def\csname b##1\endcsname{\bm{\csname##1\endcsname}}}
    \expandafter\X\next;}
 \@for\next:=#1\do{%
    \def\X##1;{\expandafter\def\csname h##1\endcsname{\widehat{\csname##1\endcsname}}}
    \expandafter\X\next;}
 \@for\next:=#1\do{%
    \def\X##1;{\expandafter\def\csname c##1\endcsname{\check{\csname##1\endcsname}}}
    \expandafter\X\next;}
 \@for\next:=#1\do{%
    \def\X##1;{\expandafter\def\csname hb##1\endcsname{\widehat{\bm{\csname##1\endcsname}}}}
    \expandafter\X\next;}
}
\greekvectors{alpha,beta,gamma,delta,epsilon,zeta,theta,kappa,lambda,mu,nu,xi,pi,rho,tau,phi,chi,psi,omega,
    Delta,Gamma,Theta,Lambda,Xi,Pi,Sigma,Phi,Psi,Omega}
    
\@tfor\next:=abcfghijlmnopqrstuvwxyzABCDGHIJLMOQSTUVWXYZ\do{%
    \def\command@factory#1{\expandafter\def\csname #1\endcsname{\mathbf{#1}} }
    \expandafter\command@factory\next
}
\@tfor\next:=abcdeghijklnopqrstuvwxyzABCDEFGHIJKLMNOPQRSTUVWXYZ\do{%
    \def\command@factory#1{\expandafter\def\csname b#1\endcsname{\mathbbm{#1}} }
    \expandafter\command@factory\next
}
\@tfor\next:=ABCDEFGHIJKLMNOPQRSTUVWXYZ\do{%
    \def\command@factory#1{\expandafter\def\csname c#1\endcsname{\mathcal{#1}} }
    \expandafter\command@factory\next
}
\@tfor\next:=abdefghijklmnopqrstuvwxyzABCDEFGHIJKLMNOPQRSTUVWXYZ\do{%
    \def\command@factory#1{\expandafter\def\csname s#1\endcsname{\mathsf{#1}} }
    \expandafter\command@factory\next
}
\@tfor\next:=abcdefghjklmnopqrstuvwxyzABCDEFGHIJKLMNOPQRSTUVWXYZ\do{
    \def\command@factory#1{\expandafter\def\csname f#1\endcsname{\mathfrak{#1}} }
    \expandafter\command@factory\next
}

\begin{document}

\begin{frontmatter}
\title{Optimal Bayesian estimation of Gaussian mixtures with growing number of components}%\thanksref{t1}
\runtitle{Optimal Bayesian estimation of Gaussian mixtures}
%\thankstext{T1}{A sample additional note to the title.}

\begin{aug}
\author[A]{\fnms{Ilsang} \snm{Ohn}\ead[label=e1,mark]{ilsang.ohn@inha.ac.kr}}
\and
\author[B]{\fnms{Lizhen} \snm{Lin}\ead[label=e2,mark]{lizhen.lin@nd.edu}}

\address[A]{Department of Statistics, Inha University, Incheon, Republic of Korea.
\printead{e1}}

\address[B]{Department of Applied and Computational Mathematics and Statistics, The University of Notre Dame, Indiana, USA.
\printead{e2}}
\end{aug}

\begin{abstract}
We study Bayesian estimation of finite mixture models in a general setup  where the number of components is unknown and allowed to grow with the sample size. An assumption on growing number of components is a natural one as the degree of heterogeneity present in the sample can grow  and new components can arise as sample size increases,  allowing  full flexibility in modeling the complexity of data. This however will lead to a high-dimensional model  which poses great challenges for estimation.  We novelly employ the idea of a sample size dependent prior in a Bayesian model and  establish a number of important theoretical results. We first show that under mild conditions on the prior, the posterior distribution concentrates around the true mixing distribution at a near optimal rate with respect to the Wasserstein distance. Under a separation condition on the true mixing distribution, we further show that a better and adaptive convergence rate can be achieved, and the number of components can be consistently estimated. Furthermore, we derive optimal convergence rates for the higher-order mixture models where the number of components diverges arbitrarily fast. In addition, we suggest a simple recipe for using Dirichlet process (DP) mixture prior for estimating the finite mixture models and provide theoretical guarantees. In particular, we provide a novel solution for  adopting the number of  clusters in a DP mixture model  as an estimate of the number of components in a finite mixture model. Simulation study and real data applications are carried out demonstrating the utilities of our method. 
\end{abstract}


\end{frontmatter}

%%%%%%%%%%%%%%%%%%%%%%%%%%%%%%%%%%%%%%%%%%%%%%%%%%%%%%%%%%%%%%%%%%%%%%%%%%%%
\section{Introduction}

Finite mixture models are powerful tools for modeling heterogeneous data, which  have been used in a wide range of applications in statistics and machine learning  including density estimation \citep{kruijer2010adaptive}, clustering \citep{fraley2002model}, document modeling \citep{blei2003latent}, image generation \citep{richardson2018gans} and designing generative adversarial networks \citep{eghbal2019mixture}, just to name a few. To date, a large number of methods, both frequentist and Bayesian,  have been proposed in the literature for various estimation problems related to  finite mixture models. Rather than listing a large body of related work here, we refer the readers to the book \cite{fruhwirth2019handbook} and a review paper \cite{mclachlan2019finite} for recent advances on finite mixture modeling. Our work focuses on the estimation of the finite mixture itself, i.e., estimating the parameters of a mixture model such as the mixing distribution and the number of mixing components, from a Bayesian perspective. %Both frequentist and Bayesian methods have gained empirical success in terms of  estimation accuracy as well as computational efficiency. 
%However, unlike their frequentist counterparts, many theoretical questions or gaps related to Bayesian estimation of finite mixtures remain open. 
Although a number of important Bayesian methods have dealt with the problem of finite mixture estimation, many interesting questions remain open.  
Most of the Bayesian work in the literature assume the number of components is either known or fixed. The minimax optimal convergence rate for estimating the mixing distribution  has not been achieved by Bayesian methods even for the fixed set up. Further,  posterior consistency on the number of components has not been established except for some special cases. This paper aims to bridge these gaps through establishing a number of new theoretical results under the general framework of finite mixture modeling with growing number of components. Allowing the number of components $k^\star$ to grow is a natural assumption and even required  in many situations, for instance, 
in topic modelling \citep{bing2018fast} and computer vision \citep{greggio2012fast}, where we expect that when the size of the sample grows so will the degree of heterogeneity present in the sample.

To have a better understanding of some of the theoretical gaps, it is important to review some of the major developments in the literature. A pioneering work on characterizing convergence rates for mixing distribution estimation in finite mixture models is due to \citet{chen1995optimal} which established a \textit{point-wise} convergence rate $C_{\nu^\star} n^{-1/4}$ for estimating the mixing distribution under the $L_1$ distance, where $n$ denotes the sample size and the $C_{\nu^\star}$ is a constant depending on the true mixing distribution $\nu^\star$.\footnote{ \citet{chen1995optimal} did not realize that the multiplicative constant $C_{\nu^\star}$  depends on the true mixing distribution, thus they argued that the rate $n^{-1/4}$ is the minimax optimal rate. This  mistake was later corrected by \citet{heinrich2018strong}.} This convergence result holds for the so-called \textit{strongly identifiable} mixtures which include the Gaussian location mixtures as special cases, and so do those stated below. \citet{nguyen2013convergence} and \citet{scricciolo2017bayesian} derived the $n^{-1/4}$ point-wise  posterior contraction rate under the second-order Wasserstein distance. \citet{ho2016strong} proved that the maximum likelihood estimator can also achieve this point-wise rate. Under the first-order Wasserstein distance, a better point-wise convergence rate $C_{\nu^\star} n^{-1/2}$ can be obtained. \citet{heinrich2018strong}, \citet{ho2020robust} and \citet{guha2019posterior} established the  $n^{-1/2}$ point-wise rate for the minimum Kolmogorov distance estimator, minimum Hellinger distance estimator and Bayesian procedure with the mixture of finite mixtures (MFM) prior, respectively. On the other hand, for the continuous mixtures where the mixing distribution admits a density function, \citet{martin2012convergence} derived a near $n^{-1/2}$ point-wise rate of the mixing density estimation for their predictive recursion algorithm \citep{newton2002nonparametric, tokdar2009consistency}.

However, due to a lack of uniformity in the constant $C_{\nu^\star}$, their analysis has been restricted to the fixed truth setup, with the number of components assumed to be either known or fixed. Also note that these  point-wise rates are not upper bounds of the actual minimax optimal rates of mixing distribution estimation, which were later derived by \citet{heinrich2018strong}. It was shown that the minimax optimal convergence rate of mixing distribution estimation for strongly identifiable mixtures, is of order $n^{-1/(4(k^\star-k_0)+2)}$, where $k^\star$ and $k_0$ denote the \textit{total} number of components and the number of \textit{well-separated} components  of the true mixing distribution, respectively. In other words, the minimax rate deteriorates with the factor $k^\star-k_0$ which can be viewed as the degree of overspecification. \citet{heinrich2018strong} also proposed a minimax optimal minimum Kolmogorov distance estimator which  however can be computationally expensive. More recently,  \citet{wu2018optimal} proposed a computationally tractable estimator called the denoised method of moments estimator for Gaussian mixture models, and showed that this estimator achieves the minimax rate. However, these minimax optimal estimators require the knowledge of the number of components $k^\star$,  which is not practical. On the other hand, no Bayesian procedure has yet been able to yield a minimax optimal rate.

In general, one does not have the prior knowledge on the number of components, and selecting an appropriate value of the number of components is a crucial step in providing accurate estimates of the true mixing distribution.  With too many components, one may suffer from large variances whereas too few components may lead to biased estimators. Also estimating the number of components may be of interest itself in practice especially when each component has a physical interpretation. A widely used approach to choose the number of components is based on a model selection criterion before estimating parameters, and a few consistent model selection criteria are available in the literature such as complete likelihood \citep{biernacki2000assessing}, the Bayesian information criteria (BIC) \citep{keribin2000consistent}, the singular Bayesian information criteria (sBIC) \citep{drton2017bayesian} and the Bayes factor \citep{chambaz2008bounds}.

A Bayesian approach is an attractive alternative due to its ability to estimate both the number of components and parameters in a unified manner. A natural strategy to infer a mixture model with an unknown number of components is to also impose a prior on the number of components $k$. By doing so, it  provides  a way of not only choosing the \textit{best} number of components (i.e., model selection), but also combining results from different mixture models with possibly varying number of components (i.e., model averaging). One notable disadvantage for such models is that  posterior computations may be challenging, since it requires developing  Monte Carlo Markov chain (MCMC) algorithms for sampling from a parameter space of varying dimensions, which often results in poor mixing or slow convergence of the Markov chain to the stationary distribution.  Several MCMC methods have been proposed  to circumvent this issue including \cite{richardson1997bayesian, stephens2000bayesian, nobile2007bayesian, miller2018mixture}. On the theoretical side, \citet{guha2019posterior} derived the $n^{-1/2}$ point-wise posterior contraction rate for this type of prior distribution. They also obtained posterior consistency of the \textit{fixed} number of components under the strong identifiability condition. Another promising approach is to use over-fitted mixtures. This approach considers a mixture model with the number of components larger than the true one and estimates the true model by discarding spurious components. \citet{rousseau2011asymptotic} studied asymptotic properties of the over-fitted mixtures and proved with a  prior on  weights of a mixture using a  Dirichlet distribution with a suitably selected hyperparameter, the spurious components vanishes asymptotically at the rate $n^{-1/2}\log^a n$ for some $a>0$ under the posterior distribution.

Our work considers a Bayesian procedure which imposes appropriate priors on both the number of components and the mixing parameters in a  general setup where the number of the mixing components is allowed to grow with sample size. We consider a general class of priors and provide assumptions on the prior on the number of components, the mixing weights as well as the atoms of the mixing distribution, that lead to optimal convergence of the posterior. Our work contributes in both methodological and theoretical development, and obtains collection of important results which can be summarized in the following. 

    \begin{enumerate}
        \item We design sample size dependent priors  and provide mild and explicit conditions on them, based on which  near-optimal posterior contraction rate of the mixing distribution estimation is derived with respect to the Wasserstein distance (\cref{thm:mixing}).  Under a separation condition on the mixing components, we further show that a better and adaptive optimal  posterior contraction rate can be obtained (\cref{thm:mixing_adap} and \cref{thm:mixing_local}). To our knowledge, this is the first minimax optimality result in the Bayesian literature.
        
        \item We derive the posterior consistency of the number of components even when the true number of components diverges (\cref{thm:k_under}). To the best of our knowledge, this is the first result on the posterior consistency of the number of components in a general setup where the true mixing distribution varies as the sample size grows.
        
        \item We propose an optimal Bayesian procedure for estimating higher-order mixture models in which the number of components diverges arbitrarily fast (\cref{thm:mixing_high}).
        
        \item  We extend our analysis to general mixture models beyond Gaussian location mixtures with growing number of components. We show that the proposed Bayesian procedure maintain the same theoretical properties even in this setup.
 
        \item We investigate some theoretical properties of the Dirichlet process (DP)  mixture models and provide a pathway for using DP models for inference of the finite mixture models (\cref{sec:dp}). The DP prior for the mixing distribution, which only generates infinite mixtures, cannot provide a meaningful posterior distribution for the number of components. We suggest a recipe for using the posterior distribution of \textit{the number of clusters}  as the estimate of the true number of components and provide theoretical guarantees.  (\cref{thm:k_upper_dp}). For mixing distribution estimation, the performance of the DP is inferior in view of the convergence rate (\cref{thm:mixing_dp}).
    \end{enumerate}

The rest of this paper is organized as follows. In \cref{sec:main}, we introduce the notation, finite Gaussian location mixture models, and the prior distribution. Then we present the main results of the paper, including optimal posterior contraction rates of the mixing distribution, and posterior consistency of the number of components. Moreover, we study theoretical properties of the proposed Bayesian procedures for estimating general mixture models. In \cref{sec:dp}, we analyze  theoretical properties of  DP mixture models for estimating the finite Gaussian mixtures. In \cref{sec:numerical}, numerical studies  including both simulation study and real data analysis are conducted for illustrating our theory. Proofs are deferred to \cref{sec:proof}. In \cref{sec:heinrich}, we provide another theoretical analysis for general mixture models with a fixed number of components under different conditions and proof techniques.

\section{Main results}
\label{sec:main}

\subsection{Notations}

We first introduce some notation that will be used throughout the paper. We denote by $\ind(\cdot)$ the indicator function. For a positive integer $n\in\bN$, we let $[n]:=\{1,2,\dots,n\}$. 
For a $d$-dimensional real vector $\x:=(x_1,\dots, x_d)\in\R^d$, we denote $\|\x\|_0:=\sum_{j=1}^d\ind(x_j\neq0)$ and $\|\x\|_\infty:=\max_{1\le j\le d}|x_j|$.  For two positive sequences $\{a_n\}_{n\in \mathbb{N}}$ and $\{b_n\}_{n\in \mathbb{N}}$,  we write $a_n\lesssim b_n$ if there exists a positive constant $C>0$ such that $a_n\le Cb_n$ for any $n\in \mathbb{N}$. Moreover, we write $a_n\gtrsim b_n$ if $b_n\lesssim a_n$ and write $a_n\asymp b_n$ if $a_n\lesssim b_n$  and  $a_n\gtrsim b_n$.   For $n$ random variables $X_1,\dots, X_n$, we use the shorthand notation $X_{1:n}:=(X_1, \dots, X_n).$   Let $\delta_\theta$ denote a Dirac measure at $\theta$. 

Let $(\fX, \cX)$ be a measurable space equipped with a Lebesgue measure $\lambda$. Let $\cP(\fX)$ be the set of all distributions supported on $\fX.$ For $G\in\cP(\fX)$, let $\P_G$ denote the probability or the expectation under the probability measure $G$. We denote by $p_G$ the probability density function of $G$ with respect to the Lebesgue measure $\lambda$ if it exists.  For $n\in\bN$, let $\P_G^{(n)}$ be the probability or the expectation under the product measure and $p_G^{(n)}$ its density function. For two probability densities $p_1$ and $p_2$, we denote by $\kl(p_1, p_2)$ the Kullback-Leibler (KL) divergence from $p_2$ to $p_1$ and by $\kl_2(p_1,p_2)$ the KL variations, i.e., $  \kl(p_1,p_2):=\int \log\del{\frac{p_1(x)}{p_2(x)}}p_1(x)\lambda(\d x)$ and $\kl_2(p_1,p_2):=\int \cbr{\log\del{\frac{p_1(x)}{p_2(x)}}}^2p_1(x)\lambda(\d x).$ 
For $\zeta>0$, a space of certain distributions $\cG$ and a distribution $G_0\in\cG$, we define a $\zeta$-KL neighborhood of $G_0$ by
    \begin{equation*}
        \cB_{\kl}(\zeta, G_0, \cG):=\cbr{G\in\cG: \kl(p_{G_0}, p_G)<\zeta^2,\kl_2( p_{G_0}, p_G)<\zeta^2}.
    \end{equation*}
For a real-valued function $f$ on $\fX$, let $\|f\|_q:=\del{\int |f(x)|^q\lambda(\d x)}^{1/q}$ for $q>0$ and $\|f\|_\infty:=\sup_{x\in\fX}|f(x)|.$ For $G\in\cP(\R)$, we denote by $m_{h}(G)$ the $h$-th moment of $G$, i.e.,  $m_{h}(G):=\int x^h \P_G(\d x).$ The $r$-th moment vector is defined by $m_{1:r}(G):=\del{ m_{1}(G), \cdots, m_r(G)}.$

\subsection{Gaussian location mixtures}

In this paper,  we  initially consider the Gaussian location mixture model in one dimension:
    \begin{equation}
    \label{eq:gm_model}
        X_1,\dots, X_n \iidsim \sum_{j=1}^k w_j\N(\theta_j, \sigma^2),
    \end{equation}
where $\theta_1,\dots, \theta_k\in \R$ are the \textit{atoms} and $(w_1,\dots, w_k)\in\Delta_k$ are the mixing \textit{weights}. Here 
we define
    \begin{equation*}
        \Delta_k:=\cbr{(w_1,\dots, w_k)\in[0,1]^k:\sum_{j=1}^kw_j=1}
    \end{equation*}
for $k\in\bN.$  We assume that the variance $\sigma^2$ is known and without loss of generality  $\sigma^2=1$. With the convolution denoted with the symbol $*$, we simply write
    \begin{equation*}
        \nu*\Phi=\sum_{j=1}^k w_j\N(\theta_j, 1)
    \end{equation*}
for the mixing distribution  $\nu:=\sum_{j=1}^kw_j\delta_{\theta_j}$, where $\Phi$ denotes the standard normal distribution. For a set $\Theta\subset \R$ and $k\in\bN$, we define the set of $k$-atomic distributions
    \begin{equation*}
        \cM_k(\Theta):=\cbr{\sum_{j=1}^kw_j\delta_{\theta_j}:(w_1,\dots, w_k)\in\Delta_k, \theta_1,\dots, \theta_k\in\Theta}.
    \end{equation*}
Note that $\cM_{k}(\Theta)\subset \cM_{k+1}(\Theta)$ for every $k\in\bN$. The parameter space is given by $\cM(\Theta):=\bigcup_{k\in\bN}\cM_k(\Theta)$.  Note that $\cM(\Theta)\subset\cP(\Theta).$

For mixture models, the Wasserstein distance is widely used as a performance measure for the mixing distribution estimation. To define the Wasserstein distance between two atomic distributions, we first define
    \begin{align*}
        \cQ(w,w'):= \cbr{(p_{jh})_{j\in[k], h\in[k']}\in [0,1]^{k\times k'}: \sum_{h=1}^{k'}p_{jh}=w_j,\sum_{j=1}^{k}p_{jh}=w_{h}', \forall j\in[k], h\in[k']},
    \end{align*}
for given two weight vectors  $w\in\Delta_k$ and $w'\in\Delta_{k'}$,  which is a set of joint distributions on $[k]\times[k']$ with marginal distributions $w$ and $w'$. For any $q\ge 1$, the $q$-th order Wasserstein distance between two atomic distributions $\nu:=\sum_{j=1}^kw_j\delta_{\theta_j}$ and $\nu':=\sum_{h=1}^{k'}w_{h}'\delta_{\theta_{h}'}$ is defined as
    \begin{equation*}
        \sW_q(\nu,\nu'):=\inf_{p\in\cQ(w,w')}\del{\sum_{j=1}^k\sum_{h=1}^{k'}p_{jh}|\theta_j-\theta'_{h}|^q}^{1/q}.
    \end{equation*}

\subsection{Prior distribution}
\label{subsec:prior}

We first assume that the true data generating process is given as $\nu^\star*\Phi$ where $\nu^\star\in \cM_{k^\star}([-L,L])$, $L>0$ for some $k^\star\in\bN$, which is the true number of mixing components. For simplicity, we write $\cM_{k}:=\cM_{k}([-L,L])$ for each $k\in\bN$ and $\cM:=\cM([-L,L]):=\cup_{k=1}^\infty \cM_k([-L,L])$. We consider a general model in which  the true mixing distribution $\nu^\star\in\cM_{k^\star}$ can vary with sample size $n$ as well as the true number of components $k^\star$ can vary with $n$. This is a critical difference from the existing Bayesian literature on mixture models which assumed a fixed true mixing distribution \citep{nguyen2013convergence, scricciolo2017bayesian, guha2019posterior}.

We assume an upper bound $\bar{k}_n\lesssim \log n/\log \log n$ on the true number of components $k^\star$. This assumption alleviates some technical difficulties, and can be justified by the following remark. Since the minimax optimal convergence rate of mixing distribution estimation for large mixtures $\nu^\star\in\cM_{k^\star}$ with $k^\star\asymp \log n/\log \log n$ is the same as the one for mixtures $\nu^\star\in\cM$ with any order, which is a  slow rate of  $\log \log n/\log n$ (See Proposition 9 of \cite{wu2018optimal}), without assuming $k^\star\lesssim \log n/\log \log n$, we may not obtain improved convergence rates.  We will also show that one can develop a Bayesian procedure that attains this minimax rate without knowing the upper bound of the true number of components. See \cref{thm:mixing_high} in \cref{subsec:higher-order}.

We now introduce our prior distribution on the finite Gaussian mixture model. The prior distribution first samples the number of components $k$ from  $\Pi(k)$ and then samples the atoms $\theta\in[-L, L]^k$ and weights $w\in\Delta_k$ from $\Pi(\theta| k)$ and $\Pi(w|k)$, respectively. Thus the prior distribution induces a distribution on $\cM=\cup_{k\in\bN}\cM_{k}$.

We impose the following conditions on the prior.

\begin{assumption*}{P}
\label{assume:prior}
Recall that $\bar{k}_n$ is the known upper bound on the true number of components. The prior distribution $\Pi$ satisfies the following conditions:
\begin{enumerate}[label=(P\arabic*)]
    \item \label{p_a1} The prior distribution on the number of components $k$ is sample size dependent. There are a constant $c_1>0$ and a sufficiently large constant $A>0$ such that for any sample size $n\in\bN$ and any $k^\circ\in\bN$, 
            \begin{equation}
            \label{eq:prior_k1}
            \frac{\Pi(k=k^\circ+1)}{\Pi(k=k^\circ)}\le c_1\e^{-A\bar{k}_n\log n}.
        \end{equation}  
    Additionally, there are constants $c_2>0$ and $c_3>0$ such that for any $n\in\bN$ and any $k^\dag\in[\bar{k}_n]$,
        \begin{equation}
          \label{eq:prior_k2}
            \Pi(k=k^\dag)\ge c_2 \e^{-(c_3\bar{k}_n\log n)k^\dag}.
        \end{equation}
        
    \item \label{p_a2} For any $k\in\bN$ and any $(w_1^0,\dots, w_k^0)\in\Delta_k$, there are positive constants $c_4$ and $c_5$ such that for any $\eta\in(0,1/k)$,
        \begin{equation}
             \Pi\del{\sum_{j=1}^{k}|w_j-w_j^0|\le \eta \big|k} \ge c_4\eta^{c_5k}.
        \end{equation}
        
    \item \label{p_a3}  For any $k\in\bN$ and any $\theta^0\in[-L,L]^k$,  there are positive constants $c_6$ and $c_7$  such that for any $\eta>0$,
        \begin{equation}
            \label{eq:prior_theta}
             \Pi\del{\max_{1\le j\le k}|\theta_j-\theta^0_j|\le \eta\big|k} \ge c_6\eta^{c_7k}.
        \end{equation}
\end{enumerate}
\end{assumption*}

In \ref{p_a1}, we require a prior distribution to heavily penalize mixture models with a large number of components, and further assume that the degree of the penalization becomes more severe of an appropriate order as sample size grows. This enables the resulting posterior distribution not to overestimate the  number of components.

\begin{remark}

The idea of using a same size dependent prior to control model complexity is not new and have frequently appeared in the Bayesian literature, e.g., prior distributions on sparsity in the multivariate normal mean model \cite{castillo2012needles}, sparsity in  nonparametric regression \cite{jiang2021variable}, the number of communities in the stochastic block model \cite{gao2020general} and the number of factors in the factor model \cite{ohn2021posterior}. 

\end{remark}

We now provide  some examples of prior distributions satisfying \cref{assume:prior}. In the following examples, the constant $A>0$ is the same as the one appearing in \labelcref{eq:prior_k1}.

\begin{example}
\label{eg:mfm1}
The mixture of finite mixture (MFM) prior considered in \cite{miller2018mixture, guha2019posterior, kruijer2010adaptive} is a hierarchical prior consisting of distributions on the number of components, the weights and the atoms. \cref{assume:prior} is met by the MFM prior with appropriate choices of each distribution. The geometric distribution with probability mass function $\Pi(k)=(1-p_n)^{k-1}p_n$ on $k$ with $p_n:=1-a\exp(-A\bar{k}_n\log n)$ for arbitrary $a>0$, satisfies \ref{p_a1}. Also, the Poisson-like distribution on $\bN$ such that $\Pi(k)=\e^{-\lambda_n} \lambda_n^{k-1}/(k-1)!$ with $\lambda_n:=a\exp(-A\bar{k}_n\log n)$ for arbitrary $a>0$ satisfies \ref{p_a1}. The Dirichlet distribution  $\Dir(\kappa_1,\dots,\kappa_k)$ on the mixing weights  with $\kappa_j\in(\kappa_0, 1)$ for every $j\in [k]$ and some $\kappa_0\in(0,1)$ satisfies \ref{p_a2}, see \cref{lem:dirichlet}.  If the prior distribution on $\theta$ behaves like a uniform distribution up to a multiplicative constant, then  \ref{p_a3} holds. 
\end{example}

\begin{example}
\label{eg:mfm3}
Consider a Binomial prior distribution on the number of components such that $k-1\sim \BINOM(\bar{k}_n-1, p_n)$ with $p_n:=a\exp(-2A\bar{k}_n\log n)$ for arbitrary $a>0$. Then this prior satisfies \labelcref{eq:prior_k1} since $\binom{\bar{k}_n-1}{k^\circ}/\binom{\bar{k}_n-1}{k^\circ-1}\le \bar{k}_n\lesssim \e^{\log \log n}$ and $1-p_n\le 1$. Also it satisfies \labelcref{eq:prior_k2} since   $1-p_n\gtrsim  1$. The MFM prior with this Binomial prior distribution satisfies \cref{assume:prior}.
\end{example}

\begin{example}
The spike and slab prior distribution on the \textit{unnormalized} weights can satisfy \ref{p_a1} and \ref{p_a2}. Suppose that we consider an over-fitted mixture model $\nu=\sum_{j=1}^{\bar{k}_n}w_j\delta_{\theta_j}$. Let $S:=\{j\in[\bar{k}_n]:w_j>0\}$, a set of indices corresponding to nonzero weights. Then we can write $\nu=\sum_{j\in S}w_j\delta_{\theta_j}$. Let $\tilde{w}\equiv (\tilde{w_j})_{j\in[\bar{k}_n]}$  be the independent random variables where $\tilde w_1$ is generated from $\GAMMA(\kappa, b)$ and the other variables, i.e., $\tilde w_2,\dots, \tilde w_{\bar{k}_n}$, are generated from a spike and slab distribution $(1-p_n)\delta_0+p_n\GAMMA(\kappa, b)$ with $p_n:=a\exp(-2A\bar{k}_n\log n)$ for  $a>0$, $b>0$ and $\kappa\in(0,1)$.  If we define the number of components as the number of nonzero elements in $\tilde{w}$ and the weights as a normalized version of $(\tilde{w}_j)_{j\in S}$, i.e., $k:=\|\tilde{w}\|_0$ and $w_j:=\tilde{w}_j/\|\tilde{w}\|_1$ for $j\in S$, then $k-1$ follows  $\BINOM(\bar{k}_n-1, p_n)$ and $(w_j)_{j\in S}$ follows  $\Dir(\kappa,\dots,\kappa)$. Thus \cref{assume:prior} holds by Examples \ref{eg:mfm1} and \ref{eg:mfm3}.
\end{example}

\subsection{Posterior concentration}
\label{subsec:post_con}

In this section, we present concentration properties of the posterior distribution $\Pi(\cdot|X_{1:n})$ defined below, with the prior given in \cref{subsec:prior} and the data from the Gaussian mixture model in \labelcref{eq:gm_model}:
    \begin{equation}
        \Pi(\d \nu|X_{1:n}):= \frac{p_{\nu*\Phi}^{(n)}(X_{1:n})\Pi(\d\nu)}{\int p_{\nu*\Phi}^{(n)}(X_{1:n})\Pi(\d\nu)}.
    \end{equation}
    
We first show that our posterior distribution does not \textit{overestimate} the number of components.

\begin{theorem}
\label{thm:k_upper}
Assume that %$\nu^\star\in \cM_{k^\star}$ where 
$k^\star\le\bar{k}_n\lesssim \log n/\log \log n$. Then with the prior distribution $\Pi$ satisfying  \cref{assume:prior}, we have
    \begin{equation}
        \inf_{\nu^\star\in \cM_{k^\star}}\P_{\nu^\star*\Phi}^{(n)}\sbr{\Pi(\nu\in \cM_{k^\star}|X_{1:n})}\to1.
    \end{equation}
\end{theorem}

%\begin{remark}Note that the condition $\nu^\star\in \cM_{k^\star}$ does not mean that $\nu^\star$ is not included in the lower order models such as $\cM_1,\dots, \cM_{k^\star-1}$ because there may be overlapped atoms or zero weights. In view of this observation, \cref{thm:k_upper} can be stated with a more precise argument as follows. Let $\breve{k}^\star$ be the \textit{smallest} number of components of the true mixing distribution $\nu^\star$ in a sense that $\nu^\star\in\cM_{\breve{k}^\star}\setminus\cM_{\breve{k}^\star-1}$. Then the conclusion of the theorem actually means that $\P_{\nu^\star*\Phi}^{(n)}\sbr{\Pi(\nu\in\cM_{\breve{k}^\star}|X_{1:n})}\to1.$\end{remark}

The following theorem shows the optimal concentration property of the posterior distribution of the mixing distribution.

\begin{theorem}
\label{thm:mixing}
Under the same assumptions of \cref{thm:k_upper}, we have
    \begin{equation}
    \label{eq:mixing_conv}
         \sup_{\nu^\star\in \cM_{k^\star}}\P_{\nu^\star*\Phi}^{(n)}\sbr{\Pi\del{\sW_1(\nu,\nu^\star)\ge M \bar{\epsilon}_{n,k^\star}\big|X_{1:n}}}=o(1)
    \end{equation}
for some universal constant $M>0$, where
    \begin{equation}
    \label{eq:rate}
         \bar{\epsilon}_{n,k^\star}:=(k^\star)^{\frac{3}{2}}\del{\frac{\log^2 n}{ n}}^{\frac{1}{4k^\star-2}}.%\frac{3k^\star-1}{2k^\star-1}
    \end{equation}
\end{theorem}

If the number of components $k^\star$ is fixed, the convergence rate in \cref{thm:mixing} is equivalent to the minimax optimal rate $n^{-1/(4k^\star-2)}$ \citep[][Proposition 7]{wu2018optimal} up to at most a logarithmic factor. % since $\bar{k}_n\lesssim \log n$. 
An additional logarithmic factor is common in the  nonparametric Bayesian literature, which often arises due to the popular  ``prior mass and testing'' proof technique which we also adopt in this paper. We refer to the papers \cite{hoffmann2015adaptive, gao2016rate} for discussions about this phenomenon.% We also adopt the ``prior mass and testing'' approach and  misses the logarithmic factor.

%The $\bar{k}_n$ factor is paid for model selection. Unlike the frequentist work \cite{wu2018optimal}, which proposes an estimation algorithm that attains the exact minimax optimal rate with the assumption that the true number of components is known, our Bayesian procedure learns the number of components. If we assume a known number of components and thus do not need a prior on the number of components,  this $\bar{k}_n$ factor can be removed.  We may be able to remove this factor using somewhat refined proof techniques without assuming the known number of components. For example, some Bayesian works on linear regression \citep{castillo2012needles, martin2017empirical} and  Gaussian directed acyclic graph models \citep{cao2019posterior, lee2019minimax} simultaneously achieved model selection consistency and the exact minimax convergence rates for parameters estimation  through a careful analysis of the likelihood ratio. We will investigate whether the same can be done for Gaussian mixture models in the near future.

%\subsection{Adaptive rates on mixing distribution and posterior consistency on the number of components}
%\label{subsec:adap}

To improve the convergence rate in \cref{thm:mixing}, one may assume that atoms are well separated and the weights are bounded away from zero. We introduce the formal definition related to this notion.
\begin{definition}
\label{def:sep}
An atomic distribution $\nu:=\sum_{j=1}^kw_j\delta_{\theta_j}$ is said to be \textit{$k_0$ $(\gamma,\omega)$-separated} for $k_0\in[k]$, $\gamma>0$ and $\omega>0$ if there exists a partition $S_1,\dots, S_{k_0}$ of $[k]$ such that
    \begin{itemize}
        \item $\abs[0]{\theta_j-\theta_{j'}}\ge \gamma$ for any $j\in S_l$, $j'\in S_{l'}$ and any $l,l'\in[k_0]$ with $l\neq l'$;
        \item $\sum_{j\in S_l}w_j\ge \omega$ for any $l\in[k_0]$.
    \end{itemize}
We let $\cM_{k, k_0, \gamma,\omega}:=\cbr{\nu\in\cM_{k}: \mbox{$\nu$ is  $k_0$ $(\gamma,\omega)$-separated}}.$
\end{definition}

In the next theorem, we derive the optimal posterior contraction rate of the mixing distribution under the separation assumption. We call this contraction rate an \textit{adaptive rate} because the result is achieved without any knowledge of the number of well-separated components $k_0$ of the true mixing distribution. 

\begin{theorem}
\label{thm:mixing_adap}
Assume that %$\nu^\star\in \cM_{k^\star, k_0, \gamma,\omega}$ where 
$k^\star\le \bar{k}_n\lesssim \log n/\log \log n$ and $\gamma\omega> M'\bar{\epsilon}_{n,k^\star}$  for a sufficiently large constant $M'>0$, where  $\bar{\epsilon}_{n,k^\star}$ is the rate defined in \labelcref{eq:rate}.
Then with the prior distribution $\Pi$ satisfying   \cref{assume:prior}, we have
    \begin{equation}
    \label{eq:mixing_conv_adap}
         \sup_{\nu^\star\in \cM_{k^\star,k_0, \gamma,\omega}}\P_{\nu^\star*\Phi}^{(n)}\sbr{\Pi\del{\sW_1(\nu,\nu^\star)\ge M\epsilon_{n,k^\star,k_0,\gamma} \big|X_{1:n}}}=o(1),
    \end{equation}
for some universal constant $M>0$, where
    \begin{equation}
    \label{eq:rate_adap}
     \epsilon_{n,k^\star,k_0,\gamma}:=(k^\star)^{\frac{6k^\star-4k_0+3}{4(k^\star-k_0)+2}}\gamma^{-\frac{2k_0-2}{2(k^\star-k_0)+1}}\del{\frac{\log^2 n}{ n}}^{\frac{1}{4(k^\star-k_0)+2}}.
    \end{equation}
\end{theorem}

\begin{remark}
\label{rmk:adap}
A nice surprise from the result of \cref{thm:mixing_adap} is that our Bayesian procedure can achieve a better convergence rate than the one in \cref{thm:mixing} without requiring any further condition on the prior distribution. This is because of fact that the condition $\gamma\omega> M'\bar{\epsilon}_n$ guarantees that the mixing distribution $\nu$ is $k_0$ $(a_0\gamma, 0)$-separated asymptotically for some constant $a_0\in(0,1)$ under the posterior distribution, provided that \cref{thm:mixing} holds.
\end{remark}

%Under the same separation condition but with the additional assumption that the number of components $k^\star$ is \emph{known}, \citet{wu2018optimal} achieved the convergence rate $C_{k^\star,\gamma}n^{-1/(4(k^\star-k_0)+2)}$ for the denoised method of moments estimator, where $C_{k^\star,\gamma}$ is some quantity depending on $k^\star$ and $\gamma$. Compared with the rate of \cite{wu2018optimal}, our convergence rate \labelcref{eq:rate_adap} has a redundant logarithmic factor due to the proof technique.  %Compared with the rate of \cite{wu2018optimal}, our convergence rate \labelcref{eq:rate_adap} has redundant factor $\bar{k}_n\log n$ due to the proof technique and the existence of the model selection step. Again the factor $\bar{k}_n$ can be removed if one assumes the number of components is known. 

In view of \cref{prop:local_sep} presented below, the convergence rate in \cref{thm:mixing_adap} is minimax optimal \citep[][Theorem 3.2]{heinrich2018strong} up to a logarithmic factor  if the model parameters $k^\star, k_0$ and $\gamma$ are fixed constants. 
\citet{heinrich2018strong} established the minimax optimal rate $n^{-1/(4(k^\star-k_0)+2)}$ of  the estimation of the mixing distribution satisfying the \textit{locally varying} condition. Namely, they showed that for fixed $k^\star\in\bN$, $k_0\in[k^\star]$ and $\nu_0\in\cM_{k_0}\setminus\cM_{k_0-1}$, it follows that
    \begin{equation}
    \label{eq:local_minimax}
        \inf_{\{\hat{\nu}\}}\sup_{\nu^\star\in\cM_{k^\star}:\sW_1(\nu^\star, \nu_0)\le \epsilon^\dag_n}\P^{(n)}_{\nu^\star *\Phi}\sbr{\sW_1(\hat{\nu},\nu^\star)} \gtrsim n^{-\frac{1}{4(k^\star-k_0)+2}},
    \end{equation}
where the infimum ranges over all possible sequences of estimators and $\epsilon^\dag_n:=n^{-1/(4(k^\star-k_0)+2)+\iota}$ for some $\iota>0$ (In fact, the above lower bound holds not only for the Gaussian location mixtures but also general mixtures satisfying the strong identifiability condition provided in \cref{def:strong}). In other words, the above minimax argument is about the true mixing distribution which does not vary globally but \textit{locally}. This locally varying condition is seemingly different from the separation condition given in \cref{def:sep}, but in fact the former is a sufficient condition of the latter. Intuitively, we can expect that the true distribution $\nu^\star\in\cM_{k^\star}$ close to  $\nu_0\in\cM_{k_0}\setminus\cM_{k_0-1}$ has at least $k_0$ well-separated components, and therefore satisfies the separation condition. We formally state this argument in the next proposition.

\begin{proposition}
\label{prop:local_sep}
Let $k_0\in\bN$ and  $\nu_0:=\sum_{j=1}^{k_0}w_{0j}\delta_{\theta_{0j}}\in\cM_{k_0}\setminus\cM_{k_0-1}$. Define
    \begin{align*}
        \gamma(\nu_0)&:=\min_{j,h\in[k_0]:j\neq h}|\theta_{0j}-\theta_{0h}|\\
        \omega(\nu_0)&:=\min_{j\in[k_0]}w_{0j}.
    \end{align*}
Let $k\in\{k_0, k_0+1, \dots\}$ and $c\in(0,1/4)$. Then we have
    \begin{align}
        \cbr{\nu\in\cM_k:\sW_1(\nu,\nu_0)<c\gamma(\nu_0)\omega(\nu_0)}
        \subset
        \cM_{k, k_0, (1-2c)\gamma(\nu_0),\frac{1-4c}{1-3c}\omega(\nu_0)}.
    \end{align}
\end{proposition}

Due to \cref{prop:local_sep}, it is clear that our Bayesian procedure is also near-optimal for the estimation of the mixing distribution under the locally varying condition. We merely state the result.

\begin{corollary}
\label{thm:mixing_local}
Assume $k^\star\le \bar{k}_n\lesssim \log n/\log \log n$. Let $k_0\in\bN$ be a fixed constant such that $k_0\le k^\star$, and let $\nu_0\in\cM_{k_0}\setminus\cM_{k_0-1}$ be a fixed distribution. Moreover, assume that the prior distribution $\Pi$ satisfies  \cref{assume:prior} . Then there exist universal constants $\tau>0$ and $M>0$ such that
    \begin{equation}
         \sup_{\nu^\star\in \cM_{k^\star}:\sW_1(\nu^\star, \nu_0)<\tau}\P_{\nu^\star* \Phi}^{(n)}\sbr{\Pi\del{\sW_1(\nu,\nu^\star)\ge M\epsilon_{n,k^\star,k_0,1} \big|X_{1:n}}}=o(1),
    \end{equation}
where $\epsilon_{n,k^\star,k_0,1}$ is  the rate in \labelcref{eq:rate} with $\gamma=1$, i.e., $ \epsilon_{n,k^\star,k_0,1}=(k^\star)^{\frac{6k^\star-4k_0+3}{4(k^\star-k_0)+2}}\del{\frac{\log^2 n}{ n}}^{\frac{1}{4(k^\star-k_0)+2}}$.

%for any  $\nu^\star\in\cM_{k^\star}$ with $\sW_1(\nu^\star, \nu_0)<\tau$ eventually.
\end{corollary}

As a byproduct, we can obtain the posterior consistency of the true number of components when the true mixing distribution $\nu^\star$ is perfectly separated, that is, $k^\star=k_0$. Note that in this case, $\nu^\star\in\cM_{k^\star}\setminus\cM_{k^\star-1}$. The following theorem states this formally.

\begin{theorem}
\label{thm:k_under}
Assume that %$\nu^\star\in \cM_{k^\star, k^\star, \gamma,\omega}$ where 
$k^\star\le \bar{k}_n\lesssim \log n/\log \log n$ and
    \begin{equation}
     \label{eq:signal}
        \gamma\omega >M'\max\{\bar\epsilon_{n,k^\star}, \epsilon_{n,k^\star,k^\star,\gamma}\}
    \end{equation}
for a sufficiently large constant $M'>0$, where  $\bar{\epsilon}_{n,k^\star}$ and $\epsilon_{n,k^\star,k^\star,\gamma}$ are the rates defined in \labelcref{eq:rate} and \labelcref{eq:rate_adap} with $k_0=k^\star$, respectively.
Then with the prior distribution $\Pi$ satisfying   \cref{assume:prior}, we have
    \begin{equation}
       \inf_{\nu^\star\in \cM_{k^\star, k^\star, \gamma,\omega}} \P_{\nu^\star*\Phi}^{(n)}\sbr{\Pi\del{\nu\in \cM_{k^\star}\setminus\cM_{k^\star-1}|X_{1:n}}}\to 1.
    \end{equation}
\end{theorem}

The condition \labelcref{eq:signal} provides a threshold for detection.  This condition plays a similar role as the beta-min condition for variable selection in linear regression \citep{castillo2012needles, martin2017empirical}.

\citet{guha2019posterior} obtained the consistency result with a similar prior distribution to ours, but their analysis is restricted to the fixed truth cases.

\subsection{Higher-order mixtures}
\label{subsec:higher-order}

In \cref{sec:main}, we have assumed that $k^\star\lesssim \log n/\log \log n.$ This assumption is justified by the minimax result for the estimation of the higher-order mixtures presented by \cite{wu2018optimal}. In this section, we prove that there is a Bayesian procedure which is similar to the one considered in \cref{sec:main}, but does not assume a known upper bound of the number of components, can attain this minimax optimality. In this case, we impose a milder condition than \ref{p_a1} on the prior.

\begin{enumerate}[label=(P\arabic*$'$)]
    \item \label{p_a1_m}  There are  constants $c_1>0$, $c_2>0$ and $b_0>0$ such that for any $n\in\bN$ and $k^\circ\in\bN$,
        \begin{equation}
          \label{eq:prior_k_m}
            \Pi(k=k^\circ)\ge c_1\e^{-(c_2\log^{b_0} n)k^\circ}.
        \end{equation}
\end{enumerate}

It is clear that any prior distribution satisfying \ref{p_a1} satisfies \ref{p_a1_m} with $b_0=2$ since $\bar{k}_n\lesssim \log n$. Also, Assumption  \ref{p_a1_m} can be met by the Poisson and geometric distribution with constant mean and success probability, respectively, which do not satisfy \ref{p_a1}.

%The above assumption ensures that the prior puts sufficient mass on the unknown number of components which diverges at a fast speed. .

The next theorem provides the convergence rate of mixing distribution estimation without any restriction on the true number of components.

\begin{theorem}
\label{thm:mixing_high}
%Assume $\nu^\star\in \cM$. 
Then with the prior distribution $\Pi$ satisfying \ref{p_a1_m}, \ref{p_a2} and \ref{p_a3}, we have
    \begin{equation}
        \sup_{\nu^\star\in \cM}\P_{\nu^\star*\Phi}^{(n)}\sbr{\Pi\del{\sW_1(\nu,\nu^\star)\ge M\frac{\log\log n}{\log n}\big|X_{1:n}}}=o(1)
    \end{equation}
for some universal constant $M>0$.
\end{theorem}

If the true mixing distribution $\nu^\star$ belongs to $\cM_{k^\star}$ with $k^\star\asymp \log n/\log \log n$, the convergence rate in the above theorem is rate-exact optimal \citep[][Theorem 5]{wu2018optimal}.

Indeed, the above result holds even when the true generating process is given by $\mu^\star*\Phi$ with $\mu^\star\in\cP([-L, L])$, which includes continuous or infinite mixtures.

\subsection{Extension to general mixture models}
\label{sec:main:general}

In this section, we extend our analysis for the Gaussian location mixture model to general mixture models with potentially growing number of components. For a mixing distribution $\nu\in\cM(\Theta)$  and a family  $\{F(\cdot, \theta):\theta\in\Theta\}$ of  distribution functions on $\R$ for $\Theta\subset \R$, we let $\nu\bullet F$ denote the distribution having a density function
    \begin{equation}
        p_{\nu\bullet F}(\cdot):= \int f(\cdot, \theta)\nu(\d\theta),
    \end{equation}
where $f(\cdot, \theta)$  stands for the probability density function of $F(\cdot, \theta)$. We call $F(\cdot, \cdot)$ and $f(\cdot, \cdot)$ a \textit{kernel distribution function} and a \textit{kernel density function}, respectively. 

We impose the following set of assumptions on the kernel distribution function.

\begin{assumption*}{F}
\label{assume:cdf}
The family $\cbr{F(\cdot, \theta):\theta\in\Theta}$ of distribution functions on $\R$ satisfies the following conditions:
\begin{enumerate}[label=(F\arabic*)]
    \item \label{Ker0} $\Theta$ is a compact subset of $\R$ with nonempty interior.
    \item \label{Ker1} There is a constant $c_1>0$  such that
        \begin{equation}
        \label{eq:ker:lip}
            \norm{f(\cdot,\theta_1)-f(\cdot,\theta_2)}_\infty
            \le c_1|\theta_1-\theta_2|
        \end{equation}
    for any $\theta_1,\theta_2\in\Theta.$ Moreover, there are constants $c_2>0$ and $r\in(0,1]$ such that 
        \begin{equation}
        \label{eq:ker:bound}
        \int p_{\nu_1\bullet F}(x)\del{\frac{ p_{\nu_1\bullet F}(x)}{p_{\nu_2\bullet F}(x)}}^r\lambda(\d x)\le c_2 %\ind\del{\frac{ p_{\nu_1\bullet F}(x)}{p_{\nu_2\bullet F}(x)} \ge \e^{1/r}}
       \end{equation}
    for any $\nu_1,\nu_2\in\cM(\Theta).$
    \item \label{Ker2} For any $k\in\bN$, there exists an estimator $\hat{M}_k$ of the moment $m_k(\nu)$ based on the sample $X_1,\dots, X_n\iidsim \P_{\nu\bullet F}$ such that
        \begin{align}
            \P_{\nu\bullet F}^{(n)}\hat{M}_k&=m_k(\nu) \label{eq:moment_unbias}\\
            \P_{\nu\bullet F}^{(n)}\del{\hat{M}_k-m_k(\nu)}^2&\lesssim \frac{1}{n}(c_3+\sqrt{k})^{2k} \label{eq:moment_quad}
        \end{align}
    for any $\nu\in\cM(\Theta)$ for some constant $c_3>0.$
\end{enumerate}
\end{assumption*}

%is a standard assumption in Bayesian mixture models \citep{scricciolo2017bayesian,guha2019posterior}, which 

Assumption \ref{Ker1} helps us establish a lower bound of a KL neighborhood of the true distribution $\nu^\star\bullet F$  by using the prior concentration conditions in \cref{assume:prior}. 
This assumption is held for a wide range of choices of kernel density function, for example, the Gaussian location family \citep{nguyen2013convergence} and more generally, location family of exponential power distributions \citep{scricciolo2011posterior}. Also, the Gaussian scale family  satisfies Assumption \ref{Ker1} as shown in the next example.
%The additional condition \ref{f_a4} is introduced to control the prior concentration of a L neighborhood of the true distribution $\nu^\star\bullet F.$ 

\begin{example}[Gaussian scale family]
%Let  $\Phi_\sigma$ and $\phi_\sigma$ denote the distribution and density functions of $\N(0,\sigma^2),$ respectively. 
Let $F$ be the kernel distribution function such that $F(x,\sigma)=\Phi_\sigma(x)$, where $\Phi_\sigma$ denotes the distribution  function of $\N(0,\sigma^2).$
Consider the Gaussian scale family $\cbr{F(\cdot, \sigma):\sigma\in[1/L,L]}$ for $L>1$. Then this family satisfies \labelcref{eq:ker:lip} since $\abs{\frac{\partial}{\partial\sigma}f(x,\sigma)\big|_{\sigma=\sigma_0}}<\infty$ for any $\sigma_0\in[1/L,L]$ and $x\in\R$. For \labelcref{eq:ker:bound}, let $r_0:=1/(2L^4)$. Fix two mixing distributions $\nu_1:=\sum_{j=1}^{k_1}w_{1,j}\delta_{\sigma_{1,j}}$ and $\nu_2:=\sum_{j=1}^{k_2}w_{2,j}\delta_{\sigma_{2,j}}$. Without loss of generality, we assume $\sigma_{i,1}<\sigma_{i,2}<\cdots<\sigma_{i,k_i}$,  for all $i=1, 2$. Then
    \begin{align*}
        \int p_{\nu_1\bullet F}(x)\del{\frac{ p_{\nu_1\bullet F}(x)}{p_{\nu_2\bullet F}(x)}}^{r_0}\lambda(\d x)
        &\lesssim \int \e^{-\frac{1}{2\sigma_{1,k_1}^2}x^2}\e^{-\frac{r_0}{2\sigma_{1,k_1}^2}x^2+\frac{r_0}{2\sigma_{2,1}^2}x^2}\lambda(\d x)\\
        &\le \int \e^{-\frac{1}{2L^2}x^2+\frac{r_0L^2}{2}x^2}\lambda(\d x)\\
        &=\int\e^{-\frac{1}{4L^2}x^2}\lambda(\d x)<\infty,
    \end{align*}
which verifies  \labelcref{eq:ker:bound}.
\end{example}

Assumption \ref{Ker2} requires the existence of the unbiased estimator of the moment of every order whose variance is bounded by certain quantity depending on the order. This condition enables us to use the theoretical tool developed in \citet{wu2018optimal}, who studied  an estimator of the mixing distribution based on the method of moments for the Gaussian location mixture model. Indeed, in the proof of our results for the Gaussian location mixture model, provided in \cref{sec:proof:main} and \cref{sec:proof}, we found the moment estimator \labelcref{eq:Gauss_moment_estimator} that satisfies Assumption \ref{Ker2}.   We give some examples that satisfy Assumption \ref{Ker2}.

\begin{example}[Gaussian location family] 
Our \cref{lem:moment_var} proves that the Gaussian location family satisfies \ref{Ker2}.
\end{example}

\begin{example}[Gaussian scale family]
Consider the Gaussian scale family $\cbr{\Phi_\sigma:\sigma>0}$. It is easy to see that the moment estimator $\hat{M}_k=\sum_{i=1}^nX_i^{k}/\E(Z^k)$ with $Z\sim\N(0,1)$ satisfies \labelcref{eq:moment_unbias} and  \labelcref{eq:moment_quad}.
\end{example}

\begin{example}[Quadratic variance exponential family (QVEF)]
An exponential family of which the variance of each distribution is at most a quadratic function of its mean is called a QVEF \citep{morris1982natural}. The class of QVEFs includes Poisson, gamma, binomial, and negative binomial distributions. If the family of  distribution functions $\cbr{F(\cdot, \theta):\theta\in\Theta}$ is a QVEF, then \ref{Ker2}  is satisfied. Indeed, Equations (8.8) and (8.6) of \citep{morris1982natural} verify \labelcref{eq:moment_unbias} and  \labelcref{eq:moment_quad}, respectively.
\end{example}

\begin{remark}
As a reviewer pointed out, Assumption \ref{Ker2} is somewhat strong and a number of mixture models do not satisfy this. For example, although the Cauchy location mixture model with $f(x,\theta)=(\pi(1+(x-\theta)^2)^{-1}$ is strongly identifiable (by \cite[][Theorem 3]{chen1995optimal}) and so can be analyzed under a different theoretical framework given in \cref{sec:heinrich}, it does not satisfy  Assumption \ref{Ker2} since the Cauchy distribution does not have finite moments of order greater than or equal to 1.
\end{remark}

Since we consider a general set of atoms $\Theta\subset\R$ rather than the interval $[-L,L]$ to include, for example, scale mixtures and exponential family mixtures, Assumption \ref{p_a3} is slightly modified 
to \cref{eq:prior_theta} being met for any $k\in\bN$ and $\theta^0\in\Theta^k$. We also assume the kernel distribution function $F(\cdot, \cdot)$ is known, i.e., no misspecification of the kernel distribution function. That is, we consider the posterior distribution denoted by $\Pi_F(\cdot|X_{1:n})$, which is defined as
    \begin{equation}
        \Pi_F(\d\nu|X_{1:n}):=\frac{p_{\nu\bullet F}^{(n)}(X_{1:n})\Pi(\d\nu)}{\int p_{\nu\bullet F}^{(n)}(X_{1:n})\Pi(\d\nu)}.
    \end{equation}
Then all the results in \cref{sec:main} can be recovered by the posterior distribution $\Pi_F(\cdot|X_{1:n})$ on the mixture model that satisfies  \cref{assume:cdf}.

\begin{theorem}
\label{thm:general}
Assume that $k^\star\le\bar{k}_n\lesssim \log n/\log \log n$. Moreover, assume that the family $\{F(\cdot, \theta):\theta\in\Theta\}$ of distribution functions on $\R$ satisfies  \cref{assume:cdf} and the prior distribution $\Pi$ satisfies  \cref{assume:prior}. Then the followings are hold:
    \begin{enumerate}[label=(\alph*)]
        \item \label{thm:general:a} It follows that
            \begin{equation}
                \inf_{\nu^\star\in \cM_{k^\star}}\P_{\nu^\star*\Phi}^{(n)}\sbr{\Pi_F(\nu\in \cM_{k^\star}|X_{1:n})}\to1;
            \end{equation}
            
        \item \label{thm:general:b} There exists an universal constant $M_1>0$ such that 
            \begin{equation}
                 \sup_{\nu^\star\in \cM_{k^\star}}\P_{\nu^\star\bullet F}^{(n)}\sbr{\Pi_F\del{\sW_1(\nu,\nu^\star)\ge M_1 \bar{\epsilon}_{n,k^\star}\big|X_{1:n}}}=o(1),
            \end{equation}
            where  $\bar{\epsilon}_{n,k^\star}$ is the rate defined in \labelcref{eq:rate};
            
        \item There exist universal constants $M_2>0$ and $M_3>0$ such that if $\gamma\omega> M_2\bar{\epsilon}_{n,k^\star}$ then 
            \begin{equation}
                 \sup_{\nu^\star\in \cM_{k^\star, k_0, \gamma,\omega}}\P_{\nu^\star\bullet F}^{(n)}\sbr{\Pi_F\del{\sW_1(\nu,\nu^\star)\ge M_3\epsilon_{n,k^\star,k_0,\gamma} \big|X_{1:n}}}=o(1),
            \end{equation}
        where $\epsilon_{n,k^\star,k_0,\gamma}$ is the rate defined in \labelcref{eq:rate_adap};
        
        \item There exist universal constants $\tau>0$ and $M_4>0$ such that
    \begin{equation}
         \sup_{\nu^\star\in \cM_{k^\star}:\sW_1(\nu^\star, \nu_0)<\tau}\P_{\nu^\star* \Phi}^{(n)}\sbr{\Pi_F\del{\sW_1(\nu,\nu^\star)\ge M_4\epsilon_{n,k^\star,k_0,1} \big|X_{1:n}}}=o(1)
    \end{equation}
    for any  fixed distribution $\nu_0\in\cM_{k_0}\setminus\cM_{k_0-1}$;

        \item \label{thm:general:e} There exists an universal constant $M_5>0$ such that if $ \gamma\omega >M_5\max\{\bar\epsilon_{n,k^\star}, \epsilon_{n,k^\star,k^\star,\gamma}\}$ then
    \begin{equation}
       \inf_{\nu^\star\in \cM_{k^\star, k^\star, \gamma,\omega}} \P_{\nu^\star*\Phi}^{(n)}\sbr{\Pi_F\del{\nu\in \cM_{k^\star}\setminus\cM_{k^\star-1}|X_{1:n}}}\to 1.
    \end{equation}
    \end{enumerate}
\end{theorem}

The proof of the theorem is straightforward, but for the sake of completeness, we provide it in \cref{sec:proof:general}.

We have considered  mixture models with the number of components $k^\star$ satisfying  $k^\star\le\bar{k}_n\lesssim \log n/\log \log n$ for general kernel functions. For higher-order mixture models with general kernel functions, we can obtain the same convergence rate as the one in \cref{thm:mixing_high}, which proves convergence rates for  higher-order Gaussian location mixtures.

\begin{theorem}
\label{thm:general_high}
Assume that the family $\{F(\cdot, \theta):\theta\in\Theta\}$ of distribution functions on $\R$ satisfies  \cref{assume:cdf} and the prior distribution $\Pi$ satisfies \ref{p_a1_m}, \ref{p_a2} and \ref{p_a3}. Then
    \begin{equation}
        \sup_{\nu^\star\in \cM}\P_{\nu^\star\bullet F}^{(n)}\sbr{\Pi_F\del{\sW_1(\nu,\nu^\star)\ge M\frac{\log\log n}{\log n}\big|X_{1:n}}}=o(1)
    \end{equation}
for some universal constant $M>0$.
\end{theorem}

\section{Dirichlet process mixtures for inference of  finite mixtures}
\label{sec:dp}

In this section, we consider Dirichlet process (DP)  prior \citep{ferguson1973bayesian} on the mixing distribution which results in an infinite mixture model-- the popular Dirichlet process (DP) mixture model. Although a DP mixture model is minimax optimal in density estimation \citep{ghosal2001entropies, ghosal2007posterior},  it suffers from a very slow convergence rate of $(\log n)^{-1/2}$ in estimating the mixing distribution of the Gaussian location mixtures as shown by \citep{nguyen2013convergence}. Their result assumes that the number of component $k^\star$ is fixed.  We consider the DP prior for the mixture distribution estimation and derive the posterior contraction rates in the most general set up by allowing the number of the components of the true mixing distribution to grow.  Further more, we adopt a natural strategy of  using the \emph{number of the clusters $T$ of the data}  to estimate the number of components and we establish posterior consistency of such a procedure. 

Note that the DP prior does not satisfy Assumption \ref{p_a1}, and thus the theorems in \cref{subsec:post_con} do not cover the case of DP prior. This section aims to separately analyze concentration properties of the posterior of the DP mixture models.

In our Gaussian location mixture setup, the DP is a distribution on \textit{infinite}-atomic distributions of the form
    \begin{equation}
    \label{eq:infinite_atomic}
        \tilde{\nu}:=\sum_{j=1}^\infty w_j\delta_{\theta_j}
    \end{equation}
where $w_1,w_2,\dots\in[0,1]$ are mixing weights such that $\sum_{j=1}^\infty w_j=1$ and $\theta_1,\theta_2,\dots\in[-L,L]$. We let $\cM_\infty$ be the set of distributions of the form \labelcref{eq:infinite_atomic}. The DP with a concentration parameter $\kappa>0$ and base distribution $H$, denoted by $\DP(\kappa, H)$, can be expressed by the following stick-breaking generation process \citep{sethuraman1994constructive}:  $w_j=E_j\prod_{h=1}^{j-1}(1-E_h)$ where $E_j\iidsim \BETA(1,\kappa)$ and $\theta_j\iidsim H.$
Since every $w_j$ generated from the above procedure is positive with probability 1, one can say that $\Pi_{\textsc{dp}}(\tilde\nu\in\cM_\infty\setminus\cM)=1$. This implies that every mixing distribution generated from the posterior of the DP mixture model has infinite number of components, therefore the posterior distribution of the number of components $k$ cannot provide any reasonable estimate of the true number of components.

One possible solution is to use an additional post-processing procedure for the posterior distribution. For example,  \citet{guha2019posterior} proposed the operator $\cT$ to infinite mixing distributions which removes weak components (in a sense that the corresponding weights are very small) and merges similar components (whose  atoms are very close) of an infinite mixing distribution so that $\cT(\tilde\nu)$ is a finite mixing distribution. They proved that for a \textit{fixed} truth $\nu^\star\in\cM_{k^\star}\setminus \cM_{k^\star-1}$, the posterior distribution of the finite mixing distribution $\cT(\tilde\nu)$ obtained after post processing concentrates to the model $\cM_{k^\star}\setminus\cM_{k^\star-1}$ under the DP prior distribution with a fixed concentration parameter.

We propose another way to infer the number of components with the DP prior. Our idea is to use the posterior distribution of the number of clusters, say $T_n$, of the data $X_{1:n}$ as an estimate of the number of components. Note that for $i\in[n]$, $X_i\iidsim \tilde{\nu}*\Phi$ can be written equivalently with the latent \textit{assignment variable} $Z_i\in\bN$ as 
    \begin{align*}
        Z_i&\iidsim w[\tilde{\nu}] :=\sum_{j=1}^{\infty}w_j\delta_{j},\\
        X_i|Z_i&\indsim \N(\theta_{Z_i},1).
    \end{align*}
where $w[\tilde{\nu}]\in\cP(\bN)$ can be viewed as the distribution on $\bN$ such that $w[\tilde{\nu}](J)=\tilde{\nu}(\{\theta_j:j\in J\})$ for any $J\subset \bN$. \textit{The number of clusters} $T_n$ is defined by
    \begin{align*}
        T_n:=T_n(Z_{1:n}):=\abs{\cbr{j\in\bN:\exists i\in[n] \mbox{ s.t. } Z_i=j}}.
    \end{align*}
    
Here we consider the joint posterior distribution of the mixing distribution $\tilde\nu$ and the latent assignment variable $Z_{1:n}$ conditioned on the data $X_{1:n}$, which is given as 
    \begin{align}
        \Pi_{\textsc{dp}}(\d\tilde{\nu}, Z_{1:n}|X_{1:n})
        := \frac{\sbr{\prod_{i=1}^n\phi(X_i-\theta_{Z_i})p_{w[\tilde{\nu}]}(Z_i)}\Pi_{\textsc{dp}}(\d\tilde\nu)}{\int\sum_{Z_{1:n}\in\bN^n}\sbr{\prod_{i=1}^n\phi(X_i-\theta_{Z_i})p_{w[\tilde{\nu}]}(Z_i)}\Pi_{\textsc{dp}}(\d\tilde\nu)},
    \end{align}
where $\phi(\cdot)$ denotes the probability density function of the standard normal distribution  and $\Pi_{\textsc{dp}}$ denotes the DP prior.

Note that the data are still assumed to be generated from the finite Gaussian mixture model $\nu^\star*\Phi$ where $\nu^\star\in\cM_{k^\star}$ for $k^\star\in\bN$ but we allow the number of components to grow at an arbitrary fast speed. Even in such general situations, we show in the following theorem that the DP prior with a suitably chosen concentration parameter can provide a nearly tight upper bound of \textit{the true number of components}.

\begin{theorem}
\label{thm:k_upper_dp}
%Assume that $\nu^\star\in \cM_{k^\star}$ with $k^\star\in\bN$. Then (n\log n)^{-1}
With the DP prior $\DP(\kappa_n, H)$, where $\kappa_n\asymp n^{-a_0} $ for $a_0>0$ and $H$ is the uniform distribution on $[-L,L]$, we have
    \begin{equation}
        \sup_{\nu^\star\in \cM_{k^\star}}\P_{\nu^\star*\Phi}^{(n)}\sbr{\Pi_{\textsc{dp}}(T_n> C k^\star|X_{1:n})}=o(1)
    \end{equation}
for some constant $C>1$ depending only on the prior distribution.
\end{theorem}

\citet{miller2013simple,miller2014inconsistency} showed that the posterior distribution of the number of clusters  does not concentrate at the true number of components if one uses the DP prior with a \textit{constant concentration parameter}. In particular, if the true data generating process is $\N(0,1)=\delta_0 *\Phi$, the posterior probability that the number of components is equal to the true number of components (i.e., 1) goes to zero \citep[][Theorem 5.1]{miller2013simple}. Our proposed sample size dependent concentration parameter resolves this inconsistency. See our simulation study in Section \ref{sec:numerical} for numerical confirmations. 

\begin{remark}
Under the MFM prior of \cite{miller2018mixture}, which is an example of prior distributions considered in \cref{sec:main}, the posterior distribution of $T_n$ is asymptotically the same as the one of $k$. \citet{miller2018mixture} proved that $|\Pi(k=k^\circ|X_{1:n})-\Pi(T_n=k^\circ|X_{1:n})|\to0$ almost surely for $k^\circ\in\bN$ as long as $\Pi(k=k')>0$ for any $k'\in[k^\circ]$. In view of this fact, the number of clusters $T_n$ can be used to infer the true number of clusters $k^\star$  even if we use the MFM prior distribution.
\end{remark}

\begin{remark}
One may wonder whether the choice of the concentration parameter $\kappa_n\asymp n^{-a_0} $ would lead to slower posterior contraction rate when the DP mixture model is used for \emph{density estimation} as a DP mixture model is commonly adopted for. It turns out that it would not.   In fact, even for $\kappa_n\asymp n^{-a_0} $, one can show that there is a universal constant $M>0$ such that
    \begin{equation*}
        \P_{\nu^\star*\Phi}^{(n)}\sbr{\Pi_{\textsc{dp}}\del{\fH(p_{\tilde\nu*\Phi}, p_{\nu^\star*\Phi})\ge M\frac{\log^c n}{\sqrt{n}}|X_{1:n}}}=o(1)
    \end{equation*}
for any $\nu^\star\in\cP([-L,L])$, for some $c>0.$ One can easily check the above result.
Following the proof of Theorem 5.1 of \citep{ghosal2001entropies} and applying \cref{lem:dirichlet}, we can see that the prior concentration near the true mixing distribution is lower bounded by $(n^{-a_0})^{c_1\log n}\gtrsim \exp(-a_0c_1\log ^2n)$ for some $c_1>0$. Thus usual prior mass and testing approach leads to the conclusion in the preceding display for estimating the density.
\end{remark}

For the estimation of the mixing distribution (of general order), we obtain the following convergence rate for the DP mixture model.%leads to a very slow convergence rate of mixing distribution estimation in general as stated in the next theorem.

\begin{theorem}
\label{thm:mixing_dp}
%Assume $\nu^\star\in \cM$. 
With the DP prior $\DP(\kappa_n, H)$, where $\exp(-c_0\log^{b_0}n)\lesssim \kappa_n\lesssim 1$ for some  $b_0>0$ and $c_0>0$ and $H$ is the uniform distribution on $[-L,L]$,  we have
    \begin{equation}
        \sup_{\nu^\star\in \cM}\P_{\nu^\star*\Phi}^{(n)}\sbr{\Pi_{\textsc{dp}}\del{\sW_1(\tilde\nu,\nu^\star)\ge M\frac{\log\log n}{\log n}\big|X_{1:n}}}=o(1)
    \end{equation}
for some universal constant $M>0$.
\end{theorem}

%The above result holds even when the true mixing distribution $\nu^\star$ is an arbitrary distribution supported on $[-L,L]$.

As one can see from our theorem above, if the true mixing distribution is of high order such that $k^\star \asymp \log n/\log \log n$, the posterior of the DP mixture model attains the minimax optimality \citep[][Theorem 5]{wu2018optimal}. However,  unlike the Bayesian procedure proposed in \cref{sec:main}, we conjecture that  posterior of the DP mixture model cannot obtain an improved convergence rate for estimating a mixing distribution when the true number of components grows slowly, say $k^\star \ll \log n/\log \log n$, because it tends to produce many redundant components.   \citet{nguyen2013convergence} analyzed the posterior of  Dirichlet process mixture endowed with a fixed concentration parameter for estimating mixing distribution  with a \textit{fixed} number of components and obtain  a slow convergence rate $(\log n)^{-1/2}$ with respect to the second-order Wasserstein distance.

\section{Numerical examples}
\label{sec:numerical}

\subsection{Simulation study}

We conduct numerical experiments to validate our theoretical findings. For the prior distribution, we use a MFM prior consisting of a Poisson distribution with mean $\lambda$ on the number of components, the Dirichlet distribution on the weights and the uniform distribution on the atoms.  For the Dirichlet distribution prior on the mixing weights, we fix its concentration parameter as a $k$-dimensional vector of 1's. For the mean parameter of the Poisson distribution, we consider the following two choices: the constant one and the one decaying with an appropriate order depending on the  sample size. We call the former \texttt{MFM\_const} and the latter \texttt{MFM\_vary}. Note that \texttt{MFM\_vary} is motivated by our theory. Python codes for reproducing the results in this section are available at \href{https://github.com/ilsangohn/bayes_mixture}{https://github.com/ilsangohn/bayes\_mixture}

\subsubsection{Inference for the mixing distribution}

We compare the performance of the proposed Bayesian method with other competitors. We consider the denoised method of moment (\texttt{DMM}) estimator proposed by \citep{wu2018optimal} and the maximum a posteriori (MAP) estimator with the Dirichlet distribution prior on the weights and the uniform distribution prior on the atoms. In the implementation of the \texttt{DMM} algorithm, we use the authors' Python codes which are available on their github repository (\href{https://github.com/albuso0/mixture}{https://github.com/albuso0/mixture}). We consider the MAP estimators of two types of mixture models: exact-fitted and over-fitted mixtures. The number of components of the exact-fitted mixture is exactly equal to the true number of components and the one of the over-fitted mixture is some upper bound $\bar{k}$ of the true number of components, in this simulation, we set $\bar{k}=2k^\star$. We call the MAP estimator of the exact-fitted mixture  \texttt{MAP\_exact} and the one of the over-fitted mixture \texttt{MAP\_over}.  We use the standard expectation-maximization (EM) algorithm to obtain MAP estimators. For the proposed Bayesian method, we use the posterior mode of the mixing distribution as an estimator. We obtain such a mode by applying the EM algorithm to mixture models with the different numbers of components and selecting the best number of components $\hat{k}$ which maximizes the posterior density of the mode. We consider the two choices of the mean parameter of the Poisson prior, $\lambda_n=\exp(-0.05\log^2n/\log\log n))$ (\texttt{MFM\_vary}) and $\lambda_n=1$ (\texttt{MFM\_const}). For all the four Bayesian methods, we set the support of the uniform distribution prior the interval $[-6,6]$ and the concentration parameter of the Dirichlet distribution prior the vector of 1's.  

We generated synthetic data sets from a Gaussian mixture model $\nu^\star*\Phi$ with $\nu^\star:=\sum_{j=1}^{k^\star}w_{j}^\star\delta_{\theta_j^\star}$. We consider the following four different cases of the true mixing distribution.
    \begin{enumerate}[label=Case \arabic*, leftmargin=1.5cm]
        \item \label{simul_case1} (Well-separated) $\theta^\star=(-3,-1,1,3)$, $w^\star=(\frac{1}{4}, \frac{1}{4}, \frac{1}{4}, \frac{1}{4})$
        \item \label{simul_case2} (Overlapped components) $\theta^\star=(-1.5,-1,1,3)$, $w^\star=(\frac{1}{4}, \frac{1}{4}, \frac{1}{4}, \frac{1}{4})$
        \item \label{simul_case3} (Weak component) $\theta^\star=(-3,-1,1,3)$,  $w^\star=(\frac{2}{5}, \frac{1}{10}, \frac{1}{4}, \frac{1}{4})$
        \item \label{simul_case4} (Higher-order) $\theta^\star=(-6,-4,-2,0,2,4,6)$, $w^\star=(\frac{1}{7}, \dots, \frac{1}{7})$ 
    \end{enumerate}
For \ref{simul_case1}, all the true components are well-separated. The true components from  \ref{simul_case2} and \ref{simul_case3} are not well-separated. In  \ref{simul_case2}, there are two close atoms and in \ref{simul_case3}, there is a weak component. \ref{simul_case4} is a higher-order mixture setup. For each setup, we let the sample size $n$ range over $\{250, 500, \dots, 2000\}.$ We repeat this data generation 50 times for each experiment and report the average of the first order Wasserstein distance between each estimator and the true mixing distribution.

\cref{fig:mixing_est} displays the average of the  the first order Wasserstein errors of the five estimators for the four cases of the data generating process. Contrary to its theoretical optimality, \texttt{DMM} performs the worst among the five estimators for all the scenarios. The performance gap of DMM to the Bayesian methods are the largest for \ref{simul_case4}. We observed that there is numerical instability of the \texttt{DMM} implementation when the number of components is larger than 5, which results in failure of computation. Thus we fixed the number of components as 4 instead of 7. This  leads to the poor performance of the method. % in estimating the higher-order mixtures,  which
The overall result seems to inconsistent to the simulation result of the \texttt{DMM} paper \citep{wu2018optimal}, which showed better or at least competitive performance of \texttt{DMM} compared to the MAP estimators. But the simulation setup is different. The authors of \citep{wu2018optimal} considered the two simulation scenarios, the first case where the true five number of components are very close to each other (indeed, $\theta^\star=(-0.236, -0.168, -0.987, 0.299, 0.150)$ and $w^\star=(0.123, 0.552, 0.010, 0.080, 0.235)$) so that the corresponding mixture density seems to be unimodal, and the second case where there are only two components of the true mixing distribution. Thus, we conjecture that \texttt{DMM} performs worse for  mixture densities with many modes. However, as one of the  reviewers pointed out, this conjecture could be potentially ungrounded. Another possible reason behind this is numerical instability. Since moments typically span several orders of magnitude, semidefinite programming in DMM procedures can suffer from numerical instabilities (see \cite{backenkohler2020bounding}).
%though the rate (modulo constants) is optimal, the asymptotic variance of the method of moments is known to be suboptimal.
This could be the reason behind the  relatively poor performance of DMM estimator.

Generally, all the four Bayesian methods performs almost similar.
For \ref{simul_case1}, the over-fitted mixture model \texttt{MAP\_over} performs worse than the other Bayesian methods, but does similar for the other three cases. For \ref{simul_case2} and \ref{simul_case3}, \texttt{MFM\_vary} tends to select the smaller mixture than the true mixture, in general, its posterior distribution is  maximized at $k=3$ which is less than the true one $k^\star=4$. Note that this does not contradict our theoretical results where we establish the consistent estimation of the number of well-separated components, which might be equal to 3 in these two cases. This leads to slightly better performance for \ref{simul_case2} where overlapped components exist and slightly worse performance for \ref{simul_case3} where weak components exist. %For the higher-order mixture case,  
Overall, knowing the true number of components does not give substantial improvement of empirical performance. Our theory for optimal estimation of the mixing distribution is built on the result of vanishing posterior probability of overestimation of the number of components (\cref{thm:k_upper}). It would be interesting to investigate the optimality of Bayesian posteriors that do not enjoy such asymptotic property, for instance, the overfitted Bayesian mixture model.

\begin{figure}[t]
    \centering   
    \begin{subfigure}[c]{0.45\textwidth}
        \includegraphics[scale=0.48]{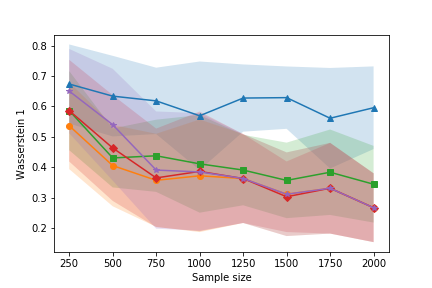}
        \subcaption{\ref{simul_case1}}
    \end{subfigure}\quad
    \begin{subfigure}[c]{0.45\textwidth}
        \includegraphics[scale=0.48]{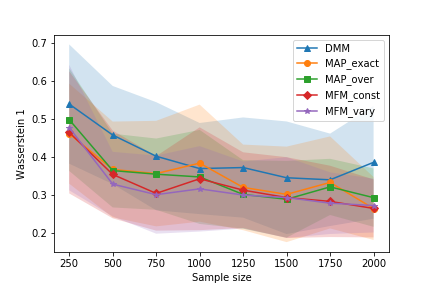}
        \subcaption{\ref{simul_case2}}
    \end{subfigure}
        \begin{subfigure}[c]{0.45\textwidth}
        \includegraphics[scale=0.48]{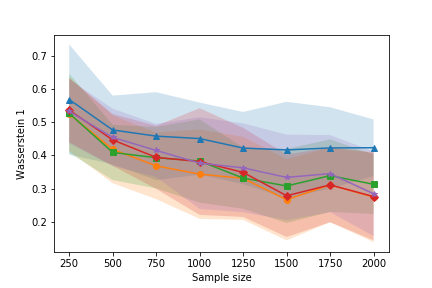}
        \subcaption{\ref{simul_case3}}
    \end{subfigure}\quad
    \begin{subfigure}[c]{0.45\textwidth}
        \includegraphics[scale=0.48]{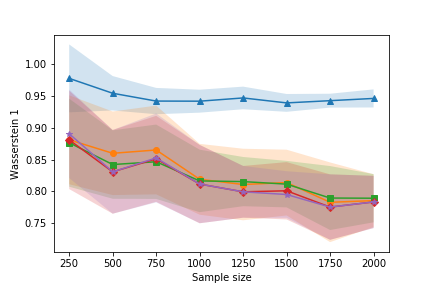}
        \subcaption{\ref{simul_case4}}
    \end{subfigure}
    \caption{The average (curve) and the standard deviation (band) of the first-order Wasserstein errors of five estimators by sample size.}
    \label{fig:mixing_est}
\end{figure}

\subsubsection{Inference for the number of components}

In this experiment, we assess the performance of the proposed Bayesian procedure and the DP mixture model with sample size dependent hyperparameters. We generated the Gaussian mixture with atoms $(-2,0,2)$ and equal weights $(1/3, 1/3, 1/3)$. Five independent data sets are generated from this Gaussian mixture model for each sample size $n\in\{50, 100, 250, 1000, 2500\}$. We compare four Bayesian methods: the two MFM models with Poisson mean parameter $\lambda_n=10\exp\del{-\frac{1}{5}\frac{\log^2n}{\log\log n}}$ (\texttt{MFM\_vary}) and $\lambda_n=1$ (\texttt{MFM\_const}) and the two DP mixtures models with concentration parameter $\kappa_n=20/n$ (\texttt{DP\_vary}) and $\kappa_n=0.4$ (\texttt{DP\_const}).  We use the uniform distribution on $[-6,6]$ for both the prior on the atoms for the MFM and the base distribution for the DP mixture.

For posterior computation for the MFM models, we employ the reversible jump MCMC algorithm of \citep{richardson1997bayesian}. For each posterior computation, we ran a single Markov chain with length 105,000. We saved every 100-th sample after a burn-in period of 5,000 samples. On the other hand, we use Neal’s Algorithm 8 \citep{neal2000markov} for non-conjugate priors to compute the posterior distributions of the DP mixtures.

\cref{fig:post_nc} presents the posterior distributions of the number of components for the two MFMs and of the number of cluster for the two DP mixtures, respectively. It clearly shows that the diminishing choices of hyperparameter advocated by our theory outperforms the constant counterparts. It is worth to notice that the posterior distribution of \texttt{DP\_vary} captures the true number of components well for large samples. It is a widely observed that the DP mixture tends to produce redundant clusters, in particular,  \citet{miller2018mixture} and \citet{guha2019posterior} observed this phenomenon in their simulation studies, however our simulation shows that a sample size dependent concentration parameter inversely related to the sample size can  circumvent this issue.

\begin{figure}[t]
    \centering   
    \begin{subfigure}[c]{0.45\textwidth}
        \includegraphics[scale=0.48]{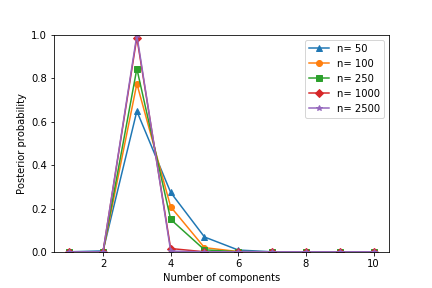}
        \subcaption{MFM with $\lambda_n=10\exp\del{-\frac{1}{5}\frac{\log^2n}{\log\log n}}$}
    \end{subfigure}\quad
    \begin{subfigure}[c]{0.45\textwidth}
        \includegraphics[scale=0.48]{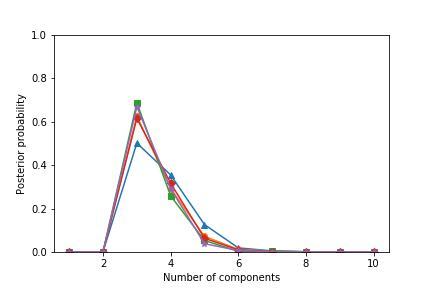}
        \subcaption{MFM with $\lambda_n=1$}
    \end{subfigure}
        \begin{subfigure}[c]{0.45\textwidth}
        \includegraphics[scale=0.48]{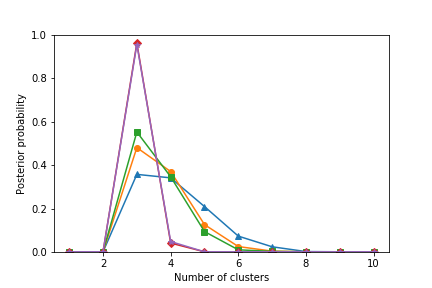}
        \subcaption{DP with $\kappa_n=20/n$ }
    \end{subfigure}\quad
    \begin{subfigure}[c]{0.45\textwidth}
        \includegraphics[scale=0.48]{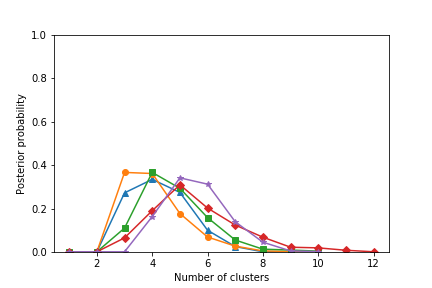}
        \subcaption{DP with $\kappa_n=0.4$ }
    \end{subfigure}
    \caption{Posterior distribution of the number of components for the MFM and of the number of clusters for the DP mixture. The true number of components is 3.}
    \label{fig:post_nc}
\end{figure}

\subsection{Real data analysis}

\subsubsection{Galaxy data}

In this section, we consider an application to the galaxy data of \citet{roeder1990density} which record velocity measurements (1,000 Km/sec) of 82 galaxies from the Corona Borealis region. This data set has been widely used as a benchmark for mixture modelling methods, e.g., \cite{escobar1995bayesian, stephens2000bayesian, nobile2007bayesian, drton2017bayesian}

To gain flexibility, we used the Gaussian location-scale mixture model instead of the Gaussian location mixture model that we have focused on. We considered the MFM and DP prior distributions. The MFM prior consists of $\Pois(\lambda)$ on the number of components, $\Dir(1,\dots,1)$ on the mixing weights, $\unif([0,40])$ on the location parameter and $\GAMMA(1,1)$ on the scale parameter. We fitted the model with the MFM prior with five different values of the mean parameter $\lambda$ of the Poisson distribution: $3,1,0.5,0.1$ and $0.01$. The base distribution of the DP for the location parameter is chosen to be the same as the MFM, but we choose the inverse gamma distribution in order to employ conjugacy. We fitted the DP mixture model with five different values of the concentration parameter $\kappa$, which are the same as $\lambda$.

\cref{fig:garaxy_ncomp} presents a posterior distribution of the number of components for the MFM and the one of the number of clusters for the DP mixture. The shape of the posterior distribution is substantially different by the choices of the hyperparameters $\lambda$ for the MFM and $\kappa$ for the DP. The posterior distribution of the number of clusters for the DP mixture model with a small concentration parameter concentrates near the value of 3, as the posterior distribution of the number of components for the MFM does. This result implies that the DP mixture can be used as a proxy of the MFM for estimating the number of components when the small concentration parameter is used, as our theory suggests. \cref{fig:garaxy_density} displays posterior predictive densities for the MFM and DP mixture as well as histogram of the galaxy data. For density estimation, it also seems that the choice of the hyperparameters is more critical than the choice among the MFM and DP mixture.

\begin{figure}
    \centering   
    \begin{subfigure}[c]{0.45\textwidth}
        \includegraphics[scale=0.48]{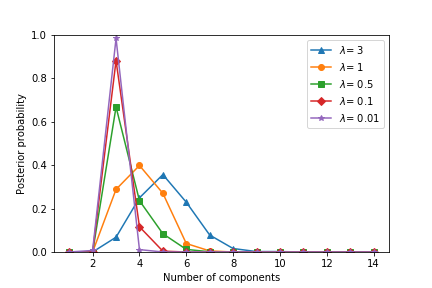}
        \subcaption{MFM}
    \end{subfigure}\quad
    \begin{subfigure}[c]{0.45\textwidth}
        \includegraphics[scale=0.48]{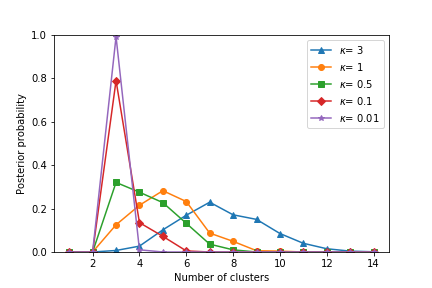}
        \subcaption{DP}
    \end{subfigure}
    \caption{Posterior distribution of the number of components for the MFM and of the number of clusters for the DP mixture for the galaxy data}
    \label{fig:garaxy_ncomp}
\end{figure}

\begin{figure}[ht]
    \centering   
    \begin{subfigure}[c]{0.45\textwidth}
        \includegraphics[scale=0.48]{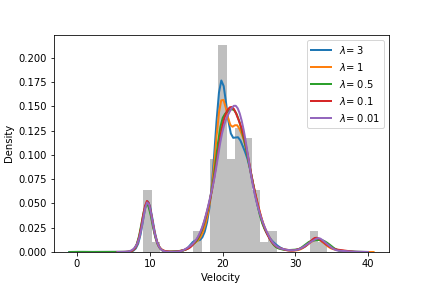}
        \subcaption{MFM}
    \end{subfigure}\quad
    \begin{subfigure}[c]{0.45\textwidth}
        \includegraphics[scale=0.48]{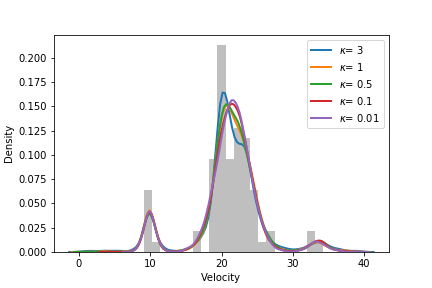}
        \subcaption{DP}
    \end{subfigure}
    \caption{Histogram of the galaxy data with posterior predictive density estimates}
    \label{fig:garaxy_density}
\end{figure}

\subsubsection{Old Faithful geyser eruption data}

In this section, we consider the Old faithful geyser eruption data which consists of two measurements, duration time and waiting time to the next eruption, taken on $n=272$ eruptions for the Old Faithful geyser in Yellowstone National Park. We analyzed the data using the multivariate Gaussian location-scale mixture model $\sum_{j=1}^kw_j\N(\theta_j, \Sigma_j)$ where $(w_1,\dots, w_k)\in \Delta_k$ are weights, $\theta_1,\dots,\theta_k$ are 2-dimensional vectors and $\Sigma_1,\dots, \Sigma_k$ are $2\times 2$ symmetric positive definite matrices. We first standardized the data and imposed the following MFM prior distribution: $k\sim\Pois(\lambda)$, $(w_1,\dots, w_k)\sim\Dir(1,\dots,1)$, $\theta_j\iidsim\unif([-2,2]^2)$ and $\Sigma_j\iidsim \textsc{wishart}(5, \frac{1}{10}I)$, where $I\in\R^{2\times2}$ denotes the identity matrix. We considered four different values of the mean parameter $\lambda$ of the Poisson distribution: $1,0.1,0.01$ and $0.001$.

\cref{fig:faithful_ncomp} displays the posterior distribution of the number of components by the value of $\lambda$. We see that  the smaller the parameter $\lambda$, the more the posterior distribution concentrates on the value of $2$, which seems to be enough to explain the data, see the scatter plot of the data in \cref{fig:faithful_density}. This result implies that our sample size dependent prior distribution may work even for the multivariate Gaussian location-scale mixture model which is much more complex than the univariate Gaussian location mixture model. \cref{fig:faithful_density} presents the posterior predictive density  by the value of $\lambda$. There is not much difference in four posterior predictive density estimates.

\begin{figure}[ht]
    \centering
    \includegraphics[scale=0.5]{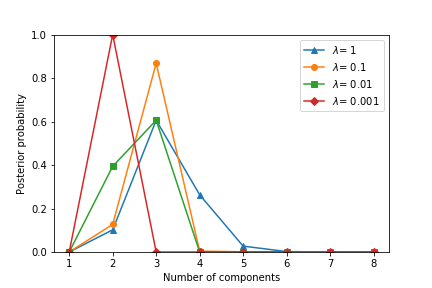}
    \caption{Posterior distribution of the number of components for the old Faithful geyser eruption data}
    \label{fig:faithful_ncomp}
\end{figure}

\begin{figure}[ht]
    \centering   
    \begin{subfigure}[c]{0.45\textwidth}
        \includegraphics[scale=0.45]{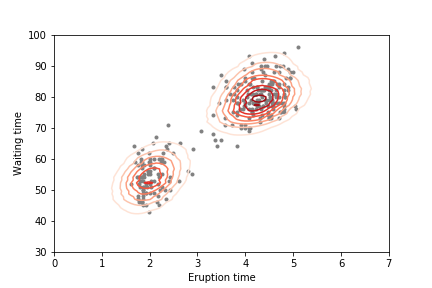}
        \subcaption{$\lambda=1$}
    \end{subfigure}\quad
    \begin{subfigure}[c]{0.45\textwidth}
        \includegraphics[scale=0.45]{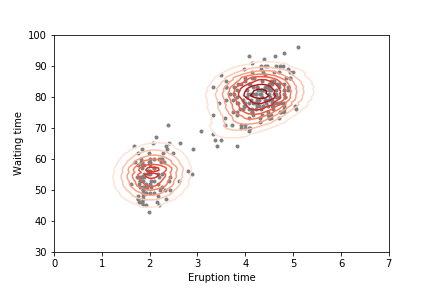}
        \subcaption{$\lambda=0.1$}
    \end{subfigure}
        \begin{subfigure}[c]{0.45\textwidth}
        \includegraphics[scale=0.45]{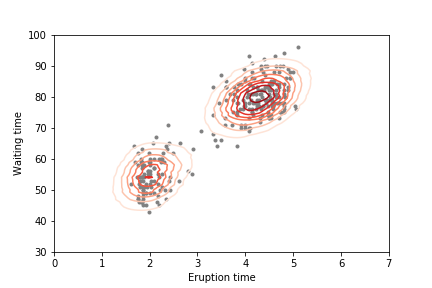}
        \subcaption{$\lambda=0.01$}
    \end{subfigure}\quad
    \begin{subfigure}[c]{0.45\textwidth}
        \includegraphics[scale=0.45]{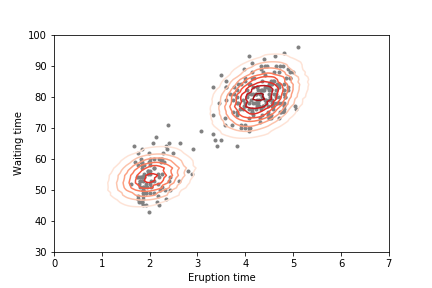}
        \subcaption{$\lambda=0.001$}
    \end{subfigure}
    \caption{Scatterplot of the old Faithful geyser eruption data with contour plot of the posterior predictive density}
    \label{fig:faithful_density}
\end{figure}

\begin{supplement} %[id-suppA]
%\sname{Supplementary}
\stitle{Supplement to ``Optimal Bayesian estimation of Gaussian mixtures with growing number of components"}
\slink[doi]{COMPLETED BY THE TYPESETTER}
\sdatatype{.pdf}
\sdescription{The  proofs of all the Theorems    are contained in  Appendix A of \citet{ohnandlin}. Analysis of general mixture models in the framework of \citet{heinrich2018strong} is provided in Appendix B of \citet{ohnandlin}.}
\end{supplement}

\section*{Acknowledgement}

We thank the Editor, the Associate Editor and two reviewers for their valuable comments.  We are in particular grateful to the two reviewers for their careful going over the technical proof of our theorems. We would like to thank Minwoo Chae for very useful comments and discussions.  We acknowledge the  generous support  of  NSF grants DMS CAREER 1654579 and DMS 2113642.

\bibliographystyle{imsart-nameyear}
\bibliography{_references}

%Add appendices
\clearpage
%\expandafter\def\expandafter\appendixpagename%
%\expandafter{\expandafter\Large\appendixpagename}

%%%%%%%%%%%%%%%%%%%%%%%%%%%%%%%%%%%%%%%%%%%%%%%%%%%%%%%%%%%%%%%%%%%%%%%%%%%%%%%%%%%%%%%%%%%%%%%%%%%%%%%%%%%%%%%%%%%%%%%%%%%%%%%%%%%%
\newpage

\setcounter{page}{1}
\renewcommand{\thepage}{S-\arabic{page}}

\crefalias{section}{appendix}

	\begin{center}
		{\Large \textsc{Supplement to "Optimal Bayesian estimation of Gaussian mixtures with growing number of components''}} \\
		\medskip 
		{Ilsang Ohn and Lizhen Lin}
		\medskip
	\end{center}

In this supplement material, we provide proofs of the results in the main text and some additional theoretical results.

\appendix

\section{Proofs}
\label{sec:proof}

We first introduce additional notations. For a real number $x\in\R$, $\floor{x}$ denotes the largest integer less than or equal to $x$ and $\ceil{x}$  the smallest integer larger than or equal to $x$. For a metric space $(\cZ, \rho),$ let $\textsf{diam}(\cZ):=\sup\cbr{\rho(z_1, z_2):z_1, z_2\in\cZ}$.

\subsection{Proofs of \cref{thm:k_upper}}

\begin{proof}[Proof of \cref{thm:k_upper}]
Let $\tilde\zeta_n:=\sqrt{\log n/n}$. We state the following well known result in the Bayesian literature (e.g., Lemma 8.1 of \citeS{ghosal2000convergence}):
    \begin{equation}
        \label{eq:denom_lbound}
                \P_{\nu^\star*\Phi}^{(n)}\del{\int\frac{p_{\nu*\Phi}^{(n)}}{p_{\nu^\star*\Phi}^{(n)}}(X_{1:n}) \Pi(\d\nu)
        \ge \e^{-2n\tilde\zeta_n^2}\Pi\del[0]{\cB_{\kl}(\tilde\zeta_n, \nu^\star*\Phi, \cM)}}
        \ge 1- \frac{1}{n\tilde\zeta_n^2}.
    \end{equation}
Due to the Lipschitz continuity of the standard normal density function and Assumptions \ref{p_a2} and \ref{p_a3}, we can follow the argument in the proof of Proposition 1 of \citeS{scricciolo2017bayesian} to obtain
   \begin{equations}
   \label{eq:prior_mass_lbound}
        \Pi\del{\cB_{\kl}(\tilde\zeta_n, \nu^\star*\Phi, \cM)}
        &\ge  \Pi\del{\cB_{\kl}(\tilde\zeta_n, \nu^\star*\Phi, \cM)|k=k^\star}\Pi(k=k^\star)\\
        &\ge \e^{-c_1k^\star\log (1/\zeta_n)}\Pi(k=k^\star)\\
        &\ge \e^{-\frac{c_1}{2}k^\star\log n}\Pi(k=k^\star)
    \end{equations}
for some constant $c_1>0$. 
Therefore,
        \begin{align*}
        \P_{\nu^\star*\Phi}^{(n)}\sbr{\Pi(\nu\notin\cM_{k^\star}|X_{1:n})}
        &=\P_{\nu^\star*\Phi}^{(n)}\sbr{\frac{\int_{\nu\notin\cM_{k^\star}} p_{\nu*\Phi}^{(n)}(X_{1:n})/ p_{\nu^\star*\Phi}^{(n)}(X_{1:n}) \Pi(\d\nu)}{\int p_{\nu*\Phi}^{(n)}(X_{1:n})/ p_{\nu^\star*\Phi}^{(n)}(X_{1:n})\Pi(\d\nu)}}\\
        &\lesssim \frac{\Pi(k>k^\star)}{\e^{-2n\tilde\zeta_n^2}\Pi\del[0]{\cB_{\kl}(\tilde\zeta_n, \nu^\star*\Phi, \cM)}} + \frac{1}{n\tilde\zeta_n^2} \\
        &\lesssim \e^{(2+c_1/2) k^\star\log n}\frac{\Pi(k> k^\star)}{\Pi(k=k^\star)} + \frac{1}{\log n}\\
        &\lesssim \e^{(2+c_1/2) k^\star\log n}\e^{-A\bar{k}_n\log n} + \frac{1}{\log n},
    \end{align*}
where we use Fubini's theorem for the first inequality. Hence if $A>c_1/2+2$, the desired result follows.
\end{proof}

\subsection{Proofs of \cref{thm:mixing}}
\label{sec:proof:mixing}

\subsubsection{Construction of test functions}

%We use a standard ``prior mass and testing'' approach to prove the convergence of the moment vector. The crucial step is to construct a test function with exponentially small error probabilities. We employ the median denoised moment estimator proposed by \citeS{wu2018optimal} to the construction of such a test function.

In this subsection, we establish the existence of test functions based on the moment estimator defined as follows. Given  $n$ i.i.d. observations $X_{1:n}$ and arbitrary $k\in\bN$ and $\eta\in(0,1)$, we first divide the sample $X_{1:n}$ to $N:=\floor{\log (2k/\eta)}\wedge n$ almost equal sized batches, say $\cX_1,\dots, \cX_N$, where each batch has $\floor{n/N}$ or $\floor{n/N}+1$ samples. Then we define for each $h\in\bN$,
    \begin{align}
     \label{eq:median_moment_estimator}
        \hat{m}^{(k,\eta)}_h:=\hat{m}^{(k,\eta)}_h(X_{1:n}):=\textsf{Median}\del{\cbr{M_{\cX_l,h}:l\in[N]}},   
    \end{align}
where we define
    \begin{equations}
      \label{eq:Gauss_moment_estimator}
    M_{\cX_l,h}:=\frac{1}{|\cX_{l}|}&\sum_{X\in\cX_{l}}H_h(X) \\
        &\mbox{ with } H_h(X):=h!\sum_{r=0}^{\floor{h/2}}\frac{(-1/2)^r}{r!(h-2r)!}X^{h-2r}
    \end{equations}
for each $l\in[N]$. Note that $H_h(x)$ is equal to the $h$-th order Hermite polynomial. For convenience, we introduce the notation $\hat{m}^{(k,\eta)}_{1:H}:=(\hat{m}^{(k,\eta)}_h)_{h\in[H]}$ for $H\in\bN$. 

In the next lemma, we derive the expectation and the upper bound of the variance of $M_{\cX_l,h}$ . 

\begin{lemma}
\label{lem:moment_var}
Suppose that each random variable in $\cX_l$ are independently generated from the distribution $\mu*\Phi$ where $\mu\in\cP([-L,L])$. Then for any $h\in\bN$, we have
    \begin{align}
    \label{eq:moment_mean}
         \E\del{M_{\cX_l,h}}=m_h(\mu)
    \end{align}
and
    \begin{align}
    \label{eq:moment_var}
         \var\del{M_{\cX_l,h}}\le \frac{1}{|\cX_{l}|} \del{c_1(L+\sqrt{h})}^{2h}.
    \end{align}
for some universal constant $c_1>0$.
\end{lemma}

\begin{proof}
 By the well known property of the Hermite polynomials that $\P_{\N(\theta,1)}H_h=\theta^h$ for any  $h\in\bN$,  we have $\P_{\mu*\Phi}H_h=\int \theta^h\mu(\d\theta)=m_{h}(\mu)$ and so \labelcref{eq:moment_mean} immediately follows. \labelcref{eq:moment_var} is Lemma 5 of \citeS{wu2018optimal}.
\end{proof}

From the above lemma, we can establish the exponential tail bound of $\hat{m}^{(k,\eta)}_{1:(2k-1)}.$   

\begin{lemma}
\label{lem:moment_tail}
Suppose that $X_1,\dots, X_n\iidsim \mu*\Phi$ where $\mu\in\cP([-L,L])$. Then for any $k\in\bN$ and $\epsilon>0$,  there is a constant $c_1>0$ depending only on $L$ such that
    \begin{align*}
        \P_{\mu*\Phi}^{(n)}& \del{\norm[0]{\hat{m}^{(k,\eta_\epsilon)}_{1:(2k-1)}-m_{1:(2k-1)}(\mu)}_\infty\ge \epsilon} 
        \le (2\e^{\frac{1}{8}})k\exp\del{-\frac{1}{8}n\cbr{(c_1 k)^{-2k+1}\epsilon^2\wedge1 }},
    \end{align*}
where  $\hat{m}^{(k,\eta_\epsilon)}_{h}$, $h\in[2k-1]$ are defined as \labelcref{eq:median_moment_estimator} with 
     \begin{equation*}
        \eta=\eta_\epsilon:=(2k)\exp\del{-(c_1k)^{-2k+1}n\epsilon^2}.
    \end{equation*}
\end{lemma}

\begin{proof}
We divide the sample to an equal sized partition $\cX_1,\dots, \cX_N$ with $N:=\floor{\log (2k/\eta_\epsilon)}\wedge n$ where $\eta_\epsilon$, which is the constant depending only on $\epsilon$, will be specified later.  Let $n_l:=|\cX_l|$ then $n_l$ is either $\floor{n/N}$ or $\floor{n/N}+1$. By \cref{lem:moment_var}, the variance of $M_{\cX_l,h}:=|\cX_{l}|^{-1}\sum_{X\in\cX_{l}}H_h(X)$ is bounded by
    \begin{equation*}
        \var(M_{\cX_l,h})\le \frac{1}{n_l}(c_0(L+\sqrt{h}))^{2h},
    \end{equation*}
for any $l\in[N]$ and any $h\in[2k-1]$  for some constant $c_0>0$. Let $C_{L,h}:=c_0(L+\sqrt{h})$ for simplicity. Then by the Chebyshev inequality with \labelcref{eq:moment_mean}, the expectation of the random variable defined as
    \begin{equation*}
        Z_{l,h}:=\ind\del{\abs{M_{\cX_l,h}-m_{h}(\mu)}< \sqrt{\frac{4}{3n_l}}\del{C_{L,h}}^{h}}
    \end{equation*}
is bounded below by
    \begin{equation}
    \label{eq:p_lh_lower}
        P_{l,h}:=\P_{\nu*\Phi}^{(n_l)} Z_{l,h}\ge \frac{3}{4},
    \end{equation}
for any $l\in[N]$ and any $h\in[2k-1]$. Now we use the well-known median trick. By definition of median and the fact that $n_l\ge \floor{n/N}\ge n/(2N)$, the moment estimator defined in \labelcref{eq:median_moment_estimator} satisfies
    \begin{equation}
    \label{eq:median_bound}
        \ind\del{\norm{\hat{m}_{h}^{(k,\eta_\epsilon)}-m_{h}(\mu)}_\infty \ge \sqrt{\frac{8N}{3n}}\del{C_{L,h}}^{h} }
        \le \ind\del{\sum_{l=1}^NZ_{l,h}\le \frac{N}{2}}
    \end{equation}
for any $h\in[2k-1]$.
By Hoeffding's inequality, the probability of the right-hand side of the preceding display is bounded as
    \begin{equation}
    \label{eq:z_lh_bound}
    \P_{\mu*\Phi}^{(n)}\del{\sum_{l=1}^NZ_{l,h}\le \frac{N}{2}}
    \le   \P_{\mu*\Phi}^{(n)}\del{\abs{\sum_{l=1}^NZ_{l,h}-\sum_{l=1}^N P_{l,h}}\ge \frac{N}{4}}
    \le \e^{-N/8},
    \end{equation}
where the first inequality is due to \labelcref{eq:p_lh_lower}. Since there is an universal constant $c_1>0$ depending only on $L$ such that $\sqrt{8/3}\del[0]{C_{L,h}}^{h}\le \sqrt{c_1k}^{2k-1}$ for any $h\in[2k-1]$, we obtain that by \labelcref{eq:median_bound}, \labelcref{eq:z_lh_bound} and the union bound,
    \begin{equation*}
        \P_{\mu*\Phi}^{(n)}\del{\norm{\hat{m}_{1:(2k-1)}^{(k,\eta_\epsilon)}-m_{1:(2k-1)}(\mu)}_\infty \ge\sqrt{\frac{N}{n}} \del{\sqrt{c_1k}}^{2k-1} }
        \le (2k)\e^{-N/8}.
    \end{equation*}
Let $\eta_\epsilon:=(2k)\exp(-(c_1k)^{2k-1}n\epsilon^2)$, then $N:=\floor{\log(2k/\eta_\epsilon)}\wedge n\le(c_1k)^{2k-1}n\epsilon^2$ and so $\epsilon \ge \sqrt{N/n}(c_1k)^{2k-1}$. Also, $N\ge \cbr{((c_1k)^{-2k+1}n\epsilon^2 )\wedge n}-1$, which completes the proof
\end{proof}
We now construct the test function for testing the true versus ``small ball'' alternatives with respect to the $\cL_\infty$ distance of moment vectors. The proof is deferred to  \cref{sec:proof:lemmas}.

\begin{lemma}
\label{lem:moment_test}
Suppose that $X_1,\dots, X_n\iidsim \mu^\star*\Phi$ where $\mu^\star\in\cP([-L,L])$. Let $\epsilon>0$, $k\in\bN$  let
    \begin{equation*}
        \cU_{\mu^\star,2\epsilon}:=\cbr{\mu\in\cP([-L,L]):\|m_{1:(2k-1)}(\mu)-m_{1:(2k-1)}(\mu^\star)\|_\infty>2\epsilon}.
    \end{equation*}
Then there is a test function $\tilde\psi:\R^n\mapsto[0,1]$ such that
    \begin{align*}
        \P^{(n)}_{\mu^\star*\Phi}\tilde\psi(X_{1:n})\vee\sup_{\mu\in\cU_{\mu^\star,2\epsilon}}\P^{(n)}_{\mu*\Phi}\sbr{1-\tilde\psi(X_{1:n})} 
        \le (2\e^{\frac{1}{8}})k\exp\del{-\frac{1}{8}n\cbr{(c_1 k)^{-2k+1}\epsilon^2\wedge1}}
    \end{align*}
for some universal constant $c_1>0$.
\end{lemma}

\begin{proof}
Let $\tilde\psi:\R^n\mapsto [0,1]$ be the function given by
     \begin{equation*}
        \tilde\psi(X_{1:n}):=\ind\del{\|\hat{m}^{(k,\eta_{n,t})}_{1:(2k-1)}-m_{1:(2k-1)}(\mu^\star)\|_\infty\ge \epsilon},  
    \end{equation*}
where $\hat{m}^{(k,\eta_{\epsilon})}_{h}$, $h\in[2k-1]$ is the moment estimator defined in  \labelcref{eq:median_moment_estimator} with
    \begin{align*}
        \eta_{\epsilon}:=(2k)\exp\del{-(c_1k)^{-2k+1}n\epsilon^2}.
    \end{align*}
for some constant $c_1>0$. Here, by \cref{lem:moment_tail}, the  constant $c_1>0$ can be chosen so that
    \begin{align*}
        \P_{\mu^\star*\Phi}^{(n)}  \tilde\psi(X_{1:n})
        \le (2\e^{\frac{1}{8}})k\exp\del{-\frac{1}{8}n\cbr{(c_1 k^\star)^{-2k+1}(2\epsilon)^2\wedge1 }}.
    \end{align*}
We just showed the exponential type \rom{1} error bound for the test function $\tilde\psi$. On the other hand, by triangle inequality, for any $\mu\in  \cU_{\mu^\star,2\epsilon}$
    \begin{align*}
        &\|\hat{m}^{(k,\eta_{\epsilon})}_{1:(2k-1)}-m_{1:(2k-1)}(\mu^\star)\|_\infty \\
        &\ge\|m_{1:(2k-1)}(\mu)-m_{1:(2k-1)}(\mu^\star)\|_\infty
         -\|\hat{m}^{(k,\eta_{\epsilon})}_{1:(2k-1)}-m_{1:(2k-1)}(\mu)\|_\infty\\
        &\ge 2\epsilon-\|\hat{m}^{(k,\eta_{\epsilon})}_{1:(2k-1)}-m_{1:(2k-1)}(\mu)\|_\infty.
    \end{align*}
Thus, the type \rom{2} error probability is bounded exponentially as
    \begin{align*}
        \sup_{\mu\in \cU_{\mu^\star,2\epsilon}}\P_{\mu*\Phi}^{(n)}\del{1-\tilde\psi(X_{1:n})}
        &=   \sup_{\mu\in \cU_{\mu^\star,2\epsilon}}\P_{\mu*\Phi}^{(n)}\del{\|\hat{m}^{(k,\eta_{\epsilon})}_{1:(2k-1)}-m_{1:(2k-1)}(\mu^\star)\|_\infty<\epsilon}  \\
        &\le    \sup_{\mu\in \cU_{\mu^\star,2\epsilon}}\P_{\mu*\Phi}^{(n)} \del{2\epsilon-\|\hat{m}_{1:(2k-1)}^{(k,\eta_{\epsilon})}-m_{1:(2k-1)}(\mu)\|_\infty<\epsilon}  \\
        &\le    \sup_{\mu\in \cU_{\mu^\star,2\epsilon}}\P_{\mu*\Phi}^{(n)}\del{\|\hat{m}^{(k,\eta_{\epsilon})}_{1:(2k-1)}-m_{1:(2k-1)}(\mu)\|_\infty> \epsilon} \\
        &\le (2\e^{\frac{1}{8}})k\exp\del{-\frac{1}{8}n\cbr{(c_1 k^\star)^{-2k^\star+1}\epsilon^2\wedge1 }}.
    \end{align*}
This completes the proof.
\end{proof}

\subsubsection{Proof of \cref{thm:mixing}}

For the proof of  \cref{thm:mixing}, we use the following moment comparison lemma to translate the mixing distribution estimation problem to the the moment vector estimation problem.

\begin{lemma}[Proposition 1 of \citetS{wu2018optimal}]
\label{lem:moment_comp}
Suppose that $\nu_1, \nu_2\in\cM_{k}([-L,L])$ for $L>0$. Let
    \begin{equation*}
        \zeta:=\norm[1]{m_{1:(2k-1)}(\nu_1)-m_{1:(2k-1)}(\nu_2)}_\infty.
    \end{equation*}
Then
    \begin{equation}
        \sW_1(\nu_1,\nu_2)\le c_1k\zeta^{\frac{1}{2k-1}}
    \end{equation}
for some constant $c_1>0$ depending only on $L$.
\end{lemma}

We are ready to give the proof.

\begin{proof}[Proof of \cref{thm:mixing}]
Fix $\nu^\star\in\cM_{k^\star}$. Let $\zeta_n:=\sqrt{\log^2 n/n}$. 
Since $\P_{\nu^\star*\Phi}^{(n)}\sbr{\Pi(\nu\in\cM_{k^\star}|X_{1:n})}\to1$ by \cref{thm:k_upper}, the proof is done if we prove that $\P_{\nu^\star*\Phi}^{(n)}\sbr[0]{\Pi(\tilde{\cU}|X_{1:n})}=o(1)$, where
    \begin{align*}
        \tilde{\cU}&:=\cbr{\nu\in\cM:\sW_1(\nu,\nu^\star)\ge M\bar\epsilon_{n,k^\star}}\bigcap\cbr{\nu\in\cM_{k^\star}}\\
        &=\cbr{\nu\in\cM_{k^\star}:\sW_1(\nu,\nu^\star)\ge M\bar\epsilon_{n,k^\star}}.
    \end{align*}
For notational simplicity, we  suppress the subscript $1$:$(2k^\star-1)$ of the moment vector and its estimator to write $m(\cdot):=m_{1:(2k^\star-1)}(\cdot)$ and $\hat{m}^{(k^\star,\eta)}:=\hat{m}^{(k^\star,\eta)}_{1:(2k^\star-1)}$. We use the notation
    \begin{equation*}
        \rho(\nu_1, \nu_2):=\norm[1]{m_{1:(2k^\star-1)}(\nu_1)-m_{1:(2k^\star-1)}(\nu_2)}_\infty
    \end{equation*}
for $\nu_1,\nu_2\in\cM.$ Let 
    \begin{equation*}
        \cU:=\cbr{\nu\in\cM_k:\rho(\nu,\nu^\star)\ge (\sqrt{M_0k^\star})^{2k^\star-1}\zeta_n},
    \end{equation*}
%where $M_{0,k^\star}:=(\sqrt{M_0k^\star})^{2k^\star-1}$
with $M_0>1$ being the constant specified later. By \cref{lem:moment_comp},
we have $\tilde{\cU}\subset\cU$ if we take $M$ so that $M\ge c_1\sqrt{M_0}$ for some constant $c_1>0$.

It remains to bound the posterior probability of $\cU$, which can be bounded as
    \begin{align*}
        \P_{\nu^\star*\Phi}^{(n)}\sbr{\Pi\del{\cU|X_{1:n}}}
        &\le \P_{\nu^\star*\Phi}^{(n)}\tilde\psi(X_{1:n}) + \P_{\nu^\star*\Phi}^{(n)}\sbr{(1-\tilde\psi(X_{1:n}))\Pi\del{\cU|X_{1:n}}\ind_{\tilde\fA}} + \P_{\nu^\star*\Phi}^{(n)}(\tilde\fA^c)
        %&\lesssim \e^{(c_8+c_{10}(c_9+A))\log^2n-A(M_{0}, k^\star, \zeta_n)}+ \log ^{-1}n.
    \end{align*}
for any test function $\tilde\psi:\R^n\mapsto [0,1]$ and measurable event $\tilde\fA\subset \R^n$. We first note that, by \labelcref{eq:denom_lbound} and \labelcref{eq:prior_mass_lbound} in the proof of \cref{thm:k_upper}, there exists a constant $c_2>0$ such that the event defined as 
    \begin{align*}
        \fA_n:=\cbr{X_{1:n}\in\R^n:\int \frac{p_{\nu*\Phi}^{(n)}}{p_{\nu^\star*\Phi}^{(n)}}(X_{1:n}) \Pi(\d\nu)
        \ge \e^{-(c_2+A)\bar{k}_n \log n}}
    \end{align*}
satisfies
    \begin{equation}
    \label{eq:prior_mass}
        \P_{\nu^\star*\Phi}^{(n)}(\fA_n)\ge 1-\frac{1}{\log n}.
    \end{equation}
On the other hand, by \cref{lem:moment_test}, there is a test function $\psi_n:\R^n\mapsto [0,1]$ such that
    \begin{align*}
        &\P^{(n)}_{\nu^\star*\Phi}\psi_n(X_{1:n})\vee\sup_{\nu\in\cU}\P^{(n)}_{\nu*\Phi}\sbr{1-\psi_n(X_{1:n})}\\
        &\le (2\e^{\frac{1}{8}})\exp\del{-\frac{1}{8}n\cbr{(c_3 k^\star)^{-2k^\star+1} (M_0k^\star)^{2k^\star-1}\del{\zeta_n/2}^2\wedge1}}
    \end{align*}
for some constant $c_3>0$. Since $k^\star\lesssim \log n\lesssim \exp(\log n)$, the preceding display is further bounded as 
    \begin{align}
    \label{eq:type1}
        \P^{(n)}_{\nu^\star*\Phi}\psi_n(X_{1:n})\vee\sup_{\nu\in\cU}\P^{(n)}_{\nu*\Phi}\sbr{1-\psi_n(X_{1:n})}
        \lesssim\exp\del{\log n-A(M_{0}, k^\star,\zeta_n)},
    \end{align}
where we denote
    \begin{align*}
        A(M_{0}, k^\star, \zeta_n)&:=
        \frac{1}{8}n\sbr{\cbr{(M_0/c_3)^{2k^\star-1}(\zeta_n/2)^2}\wedge1 }
    \end{align*}
for notational simplicity. This implies that
    \begin{equations}
    \label{eq:type2_on_A}
    &\P_{\nu^\star*\Phi}^{(n)}\sbr{(1-\psi_n(X_{1:n}))\Pi\del{\cU|X_{1:n}}\ind_{\fA_n}}\\
    &\le \e^{(c_2+A)\bar{k}_n\log n}\P_{\nu^\star*\Phi}^{(n)}\sbr{(1-\psi_n(X_{1:n}))\int_{\cU}\frac{p_{\nu*\Phi}^{(n)}}{p_{\nu^\star*\Phi}^{(n)}}(X_{1:n}) \Pi(\d\nu)}\\
     &= \e^{(c_2+A)\bar{k}_n\log n}\int_{\cU}(1-\psi_n(x_{1:n}))p_{\nu*\Phi}^{(n)}(x_{1:n})\lambda(\d x_{1:n}) \Pi(\d\nu)\\
     &\le \e^{(c_2+A)\bar{k}_n\log n}\sup_{\nu\in \cU}\P_{\nu*\Phi}^{(n)}(1- \psi_n(X_{1:n}))\\
     &\le \e^{c_4(c_2+A)\log^2n-A(M_{0}, k^\star, \zeta_n)}
    \end{equations}
for some constant $c_4>0$, where the equality follows from Fubini's theorem.

Consequently, by  \labelcref{eq:prior_mass}, \labelcref{eq:type1}, and \labelcref{eq:type2_on_A}, we obtain
    \begin{align*}
        \P_{\nu^\star*\Phi}^{(n)}\sbr{\Pi\del{\cU|X_{1:n}}}
        &\le \P_{\nu^\star*\Phi}^{(n)}\psi_n(X_{1:n}) + \P_{\nu^\star*\Phi}^{(n)}\sbr{(1-\psi_n(X_{1:n}))\Pi\del{\cU|X_{1:n}}\ind_{\fA_n}} + \P_{\nu^\star*\Phi}^{(n)}(\fA_n^c)\\
        &\lesssim \e^{(1+c_4(c_2+A))\log^2n-A(M_{0}, k^\star, \zeta_n)}+ \log ^{-1}n.
    \end{align*}
Note that for any $M_0$ such that $M_0>c_3$, we have
    \begin{align*}
        A(M_{0}, k^\star, \zeta_n)&\ge \frac{1}{8} n\cbr{\del{\frac{M_0}{c_3}\frac{\log^2n}{4n}}\wedge1 }\\
        &\ge c_5 M_0\log^2n
    \end{align*}
for some constant $c_5>0$, where the second inequality is due to that $\log^2 n/n=o(1)$. Hence the posterior probability of $\cU$  goes to zero if we choose $M_{0}$ such that $M_{0}>\max\{(1+c_4(c_2+A))/c_5,c_3,1\}$.
\end{proof}

\subsection{Proof of  \cref{thm:mixing_adap}}

\begin{proof}[Proof of \cref{thm:mixing_adap}]
To avoid confusion, we denote by $\bar{M}$ instead of $M$ the sufficiently large constant appearing in \labelcref{eq:mixing_conv} and we let $M$ be the constant appearing in \labelcref{eq:mixing_conv_adap}. If $M\epsilon_{n,k^\star,k_0,\gamma}\ge \bar{M}\bar\epsilon_{n,k^\star}$, the result follows trivially from \cref{thm:mixing}, so we assume throughout that  $M\epsilon_{n,k^\star,k_0,\gamma}< \bar{M}\bar\epsilon_{n,k^\star}$.

Fix $\nu^\star\in\cM_{k^\star, k_0, \gamma,\omega}$. Since
    \begin{align*}
        \P_{\nu^\star*\Phi}^{(n)}\sbr{\Pi(\nu\notin\cM_{k^\star}|X_{1:n})}=o(1) & \mbox{ by \cref{thm:k_upper}}\\
        \P_{\nu^\star*\Phi}^{(n)}\sbr{\Pi(\sW_1(\nu, \nu^\star)\ge \bar{M}\bar{\epsilon}_{n,k^\star}|X_{1:n})}=o(1) & \mbox{ by \cref{thm:mixing}}, 
    \end{align*}
we will be done with the proof if we can show that 
    \begin{equation*}
        \P_{\nu^\star*\Phi}^{(n)}\sbr{\Pi\del{\cbr{\nu\in\cM_{k^\star}:\bar{M}\bar{\epsilon}_{n,k^\star} >\sW_1(\nu, \nu^\star)\ge M\epsilon_{n,k^\star,k_0,\gamma}}|X_{1:n}}}=o(1).
    \end{equation*}
Let $\nu:=\sum_{j=1}^{k^\star}w_j\delta_{\theta_j}$ be the mixing distribution satisfying $W_1(\nu, \nu^\star)\le \bar{M}\bar\epsilon_{n,k^\star}$ for the true mixing distribution $\nu^\star:=\sum_{j=1}^{k^\star}w_j^\star\delta_{\theta_j^\star}$. Since $\nu^\star$ is $k_0$ $(\gamma, \omega)$-separated, there is a partition $(S_l:l\in[k_0])$ of $[k^\star]$ such that $\abs[0]{\theta_j-\theta_{j'}}\ge \gamma$ for any $j\in S_l$, $j'\in S_{l'}$ and any $l,l'\in[k_0]$ with $l\neq l'$ and $\sum_{j\in S_l}w_j\ge \omega$ for any $l\in[k_0]$. For each $h\in[k^\star]$, let $j_h^*=\argmin_{j\in[k^\star]}|\theta_j-\theta_h^\star|$. Note that for any $l\in[k_0]$,
    \begin{align*}
       \sW_1(\nu, \nu^\star)\ge \sum_{h\in S_l}w_h^\star|\theta_{j_h^*}-\theta_h^\star|\ge \omega\min_{h\in S_l}|\theta_{j_h^*}-\theta_h^\star|.
    \end{align*}
We now suppose that the assumption $\gamma\omega> M'\bar\epsilon_{n,k^\star}$ holds with $M':=\frac{1}{c}\bar{M}$ for some constant $c$ less than $1/2$. Then
    \begin{align*}
        \min_{h\in S_l}|\theta_{j_h^*}-\theta_h^\star|\le \sW_1(\nu, \nu^\star)/\omega\le  \bar{M}\bar\epsilon_{n,k^\star}/\omega\le c\gamma.
    \end{align*}
That is, for any $l\in[k_0]$, there is  $h\in S_l$ such that $\theta_h^\star$ is close to some atom of $\nu$ within distance $\gamma/c$. Hence the mixing distribution $\nu$ is $k_0$ $((1-2c)\gamma, 0)$ separated. Let $S:=\cbr{\theta_j:j\in[k^\star]}\cup\cbr{\theta_j^\star:j\in[k^\star]}$. Then each element in $S$ is $(1-2c)\gamma$ away from at least $2(k_0-1)$ elements in $S$. Therefore we can invoke Proposition 4 of \citetS{wu2018optimal} which states that if $\nu_1$ and $\nu_2$ are supported on a set of $r$ atoms in $[-L,L]$ and each atom is at least $\tilde\gamma$ away from all but at most $r'$ atoms, then
    \begin{equation*}
     \sW_1(\nu_1,\nu_2)\le c_1r\del{\frac{r4^{r-1}}{\tilde\gamma^{r-r'-1}}\norm{m_{1:(r-1)}(\nu_1)-m_{1:(r-1)}(\nu_2)}_\infty}^{\frac{1}{r'}},
    \end{equation*}
for some constant $c_1>0$ depending only on $L$. By the preceding disply with $r=2k^\star$, $r'=2k^\star-1-2(k_0-1)=2(k^\star-k_0)+1$  and $\tilde\gamma=(1-2c)\gamma$, we have for sufficiently large $M>0$
    \begin{align*}
        &\cbr{\nu\in\cM_{k^\star}:\bar{M}\bar{\epsilon}_{n,k^\star} > W_1(\nu, \nu^\star)\ge M\epsilon_{n,k^\star,k_0,\gamma}}\\
        &\subset\cbr{\nu\in\cM_{k^\star}:\norm[0]{m_{1:(2k^\star-1)}(\nu)-m_{1:(2k^\star-1)}(\nu^\star)}_\infty\ge (\sqrt{M_0k^\star})^{2k^\star-1}\zeta_n},
    \end{align*}
with $M_0>1$ being sufficiently large, and $\zeta_n:=\sqrt{\log n^2/n}$. The only remaining part of the proof is to bound the posterior probability of the right-hand side of the preceding display, and this is shown in the proof of \cref{thm:mixing}.
\end{proof}

\subsection{Proof of  \cref{prop:local_sep}}
 
\begin{proof}[Proof of \cref{prop:local_sep}]
We set $\gamma:=\gamma(\nu_0)$ and $\omega:=\omega(\nu_0)$ for short. Suppose that $\nu:=\sum_{j=1}^{k}w_j\delta_{\theta_j}\in\cM_{k}$ satisfies $\sW_1(\nu, \nu_0)<c\gamma\omega$. Since $\nu_0$ is $k_0$ $(\gamma, \omega)$-separated, by the similar argument in the proof of \cref{thm:mixing_adap}, we have that for every  $h\in[k]$, 
    \begin{align*}
        |\theta_{j_h^*}-\theta_{0h}|\le \sW_1(\nu, \nu_0)/\omega\le  c\gamma,
    \end{align*}
where we define $j_h^*=\argmin_{j\in[k]}|\theta_j-\theta_h^\star|$. Thus, $\nu$ is $k_0$ $((1-2c)\gamma, 0)$ separated. Moreover, since $|\theta_{j_h^*}-\theta_{0l}|\ge |\theta_{0h}-\theta_{0l}|-|\theta_{j_h^*}-\theta_{0h}|\ge(1-c)\gamma>c\gamma$ for any $l\neq h,$ the indices $j_{1}^*, \dots, j_{k_0}^*$ are distinct. Thus there is a partition $S_1,\dots,S_{k_0}$ of $[k]$ such that  $|\theta_j-\theta_{j'}|\ge(1-2c)\gamma$ for any $j\in S_h$, $j'\in S_{h'}$ and any $h,h'\in[k_0]$ with $h\neq h'$ and $j_h^*\in S_h$ for any $h\in[k_0]$. Let $(p^*_{jh})_{j\in[k],h\in[k_0]}\in\cQ((w_j)_{j\in[k]}, (w_{0j})_{j\in[k_0]})$ be the optimal coupling such that $\sW_1(\nu,\nu_0)=\sum_{j=1}^k\sum_{h=1}^{k_0}p_{jh}|\theta_j-\theta_{0h}|$. Then for any $h\in[k_0]$, we have
    \begin{align*}
        c\gamma\omega>\sW_1(\nu,\nu_0) &\ge \sum_{j=1}^kp^*_{jh}|\theta_j-\theta_{0h}|\\
        &= \sum_{j\in S_h}p^*_{jh}|\theta_j-\theta_{0h}| + \sum_{j \notin S_h}p^*_{jh}|\theta_j-\theta_{0h}|\\
        &\ge 0+ \del{w_{0h}-\sum_{j\in S_h}p^*_{jh}}(1-3c)\gamma,
    \end{align*}
where the last inequality follows from that $|\theta_j-\theta_{0h}|\ge |\theta_j-\theta_{j_h^*}|- |\theta_{j_h^*}-\theta_{0h}|\ge (1-2c)\gamma-c\gamma$ for any $j \notin S_h.$
Hence,
    \begin{align*}
        \sum_{j\in S_h}w_{j}\ge \sum_{j\in S_h}p^*_{jh}\ge w_{0h}-\frac{c}{1-3c}\omega\ge \frac{1-4c}{1-3c}\omega,
    \end{align*}
which completes the proof.
\end{proof}

\subsection{Proof of \cref{thm:k_under}}

\begin{proof}[Proof of \cref{thm:k_under}]
Assume that $\nu:=\sum_{j=1}^kw_j\delta_{\theta_j}\in\cM_{k}$ with $k<k^\star$. Then there exists an index  $h^*\in[k^\star]$ such that
    \begin{equation*}
        |\theta_j-\theta_{h^*}^\star|\ge \min_{h\in[k^\star]:h\neq j^*}|\theta_j-\theta_h^\star|
    \end{equation*}
for any $j\in[k]$, which implies that
    \begin{align*}
        2|\theta_j-\theta_{h^*}^\star| 
        &\ge |\theta_j-\theta_{h^*}^\star| +  \min_{h\in[k^\star]:h\neq j^*}|\theta_j-\theta_h^\star|\\
        &\ge \min_{h,l\in[k^\star]:h\neq l}|\theta_l^\star-\theta_h^\star|.
    \end{align*}
Therefore, for the optimal coupling $(p^*_{jh})_{j\in[k], h\in[k^\star]}\in\cQ((w_j)_{j\in[k]},(w_j^\star)_{j\in[k^\star] })$, we have
    \begin{align*}
        W_1(\nu,\nu^\star)&=\sum_{j=1}^k\sum_{h=1}^{k^\star}p^*_{jh}|\theta_j-\theta_h^\star|\\
        &\ge \sum_{j=1}^kp^*_{jh^*}|\theta_j-\theta_{h^*}^\star|\\
        &\ge \frac{1}{2}w^\star_{h^*}|\theta_l^\star-\theta_h^\star|\ge \frac{\gamma\omega}{2}.
    \end{align*}
Since $\gamma\omega > M'\epsilon_{n,k^\star,k_0,\gamma}$ for some large constant $M'>0$ by assumption, we have
    \begin{align*}
        \cbr{\nu\in\cM_k}&\subset  \cbr{\nu\in\cM: \sW_1(\nu,\nu^\star)\ge \gamma\omega/2}\\
        &\subset \cbr{\nu\in\cM: \sW_1(\nu,\nu^\star)\ge M'\epsilon_{n,k^\star,k_0,\gamma}/2}.
    \end{align*}
The proof is complete by \cref{thm:mixing_adap}.
\end{proof}

\subsection{Proof of \cref{thm:mixing_high}}

We invoke the following moment comparison lemma for general distributions,
which is a direct consequence of Lemma 24 of \citetS{wu2018optimal}, in which the upper bound of the Wasserstein distance is given with the $\cL_2$ distance between moments vectors.

\begin{lemma}[ Lemma 24 of \citetS{wu2018optimal}]
\label{lem:moment_high}
Let $\mu_1, \mu_2\in\cP([-L,L])$ and  $r\in\bN$. Then
    \begin{align*}
        \sW_1(\mu_1,\mu_2)\le c_1\cbr{\frac{1}{r+1}+\sqrt{r}(c_2)^r\|m_{1:r}(\mu_1)-m_{1:r}(\mu_2)\|_\infty}.
    \end{align*}
for some constants $c_1>0$ and $c_2>1$ depending only on $L.$
\end{lemma}

\begin{proof}[Proof of \cref{thm:mixing_high}]
Let $\tilde\xi_n:=n^{-1/3}\log^{1/2}n$ and $\xi_n:=n^{-2/3}\log ^{-2}n$ so that $\xi_n\log ^2(1/\xi_n)\lesssim \tilde\xi_n^2$.
Following the proof of \cref{thm:k_upper}, we have
    \begin{align*}
        D_n:=\int\frac{p_{\nu*\Phi}^{(n)}}{p_{\nu^\star*\Phi}^{(n)}}(X_{1:n}) \Pi(\d\nu)
        &\ge \e^{-2n\tilde\xi_n^2}\Pi\del[0]{\cB_{\kl}(\tilde\xi_n, \nu^\star*\Phi, \cM)}\\
        &\ge  \e^{-2n\tilde\xi_n^2}\Pi\del{\nu \in\cM:\sW_2^2(\nu,\nu^\star)\le c_1\xi_n}
    \end{align*}
with $\P_{\nu^\star*\Phi}^{(n)}$-probability at least $ 1- 1/n\tilde\xi_n^2$ for some constant $c_1>0$.
Let $R:=\ceil{4L/\sqrt{c_1\xi_n}}$ and $B_1, \dots, B_R$ be a partition of $[-L,L]$ such that $\textsf{diam}(B_j)\le \sqrt{c_1\xi_n}/2$. By Lemma 3 of \citeS{gao2016posterior},
    \begin{align*}
        \sW_2(\nu,\nu^\star)\le \frac{\sqrt{c_1\xi_n}}{2}+ 2L\del{\sum_{j=1}^R|\nu(B_j)-\nu^\star(B_j)|}^{1/2}.
    \end{align*}
which implies that
    \begin{align*}
        &\Pi\del{\sW_2^2(\nu,\nu^\star)\le c_1\xi_n}\\
        &\ge
        \Pi\del{\sum_{j=1}^R|\nu(B_j)-\nu^\star(B_j)|\le \frac{c_1\xi_n}{16L^2}} \\
        &\ge\Pi\del{\sum_{j=1}^R|\nu(B_j)-\nu^\star(B_j)|\le \frac{c_1\xi_n}{16L^2}\big|\cbr{k=R}\cap\cE} \Pi(\cE|k=R)\Pi(k=R),
    \end{align*}
where $\cE$ denotes the event such that each $B_j$ contains exactly one atom of $\nu$. By \ref{p_a1_m}, $-\log \Pi(k=R)\gtrsim R\log^{b_0}n\gtrsim n^{1/3}\log^{b_0}n$ and by \ref{p_a3}, $-\log \Pi(\cE|k=R)\gtrsim -R\log (\xi_n^{-1})\gtrsim n^{1/3}\log n$. By \ref{p_a2}, 
    \begin{align*}
        -\log \Pi\del{\sum_{j=1}^R|\nu(B_j)-\nu^\star(B_j)|\le \frac{c_1\xi_n}{8L^2}\big|\cbr{k=R}\cap\cE}
        \gtrsim n^{1/3}\log n.
    \end{align*}
Combining the results, we arrive at
    \begin{align*}
        \P_{\nu^\star*\Phi}^{(n)}\del{D_n \ge \e^{-c_2n^{1/3}\log^{1\vee b_0} n}}\ge 1-\frac{1}{n^{1/3}\log n}
    \end{align*}
for some constant $c_2>0$.

Let $\hat{k}$ be the positive integer such that $ \hat{k}\asymp  \log n/\log\log n$ but $ 2\hat{k}-1\le\frac{1}{2}\log n/\log\log n$. By applying \cref{lem:moment_high} with $r=2\hat{k}-1$, if $M$ is sufficiently large, we obtain
    \begin{align*}
        &\cbr{\nu\in\cM:\sW_1(\nu,\nu^\star)\ge M \frac{\log \log n}{\log n}}\\
        &\subset \cbr{\nu\in\cM:\|m_{1:(2\hat{k}-1)}(\nu)-m_{1:(2\hat{k}-1)}(\nu^\star)\|_\infty\ge M'(\hat{k})^{-1/2} c_3^{-\hat{k}}\frac{\log \log n}{\log n}} \\
        &\subset \cbr{\nu\in\cM:\|m_{1:(2\hat{k}-1)}(\nu)-m_{1:(2\hat{k}-1)}(\nu^\star)\|_\infty\ge M'c_3^{-\hat{k}} \log ^{-2}n}
    \end{align*}
for some constant $M'>0$ depending only on $M$ and $L$, and some  $c_3>1$  depending only on $L$.  Following the proof of \cref{thm:mixing}, it suffices to show that
    \begin{align*}
    \del{ \hat{k}^{c_3\hat{k}}\vee \e^{c_2n^{1/3}\log^{1\vee b_0}  n}} \exp\del{-c_4(M')^2n\hat{k}^{-2\hat{k}+1}c_3^{-2\hat{k}} \log ^{-4}n}=o(1)
    \end{align*}
for some constants $c_3, c_4>0$. Note that $(u\hat{k})^{-2\hat{k}+1}\gtrsim n^{-1/2}$ for any constant $u>0$, thus the preceding display holds clearly.
\end{proof}

\subsection{Proofs for \cref{sec:main:general}}
\label{sec:proof:general}

Let $\fH(p_1, p_2)$ denote the Hellinger distance between two probability densities $p_1$ and $p_2$ with respect to the Lebesgue measure $\lambda$, which is defined as $\fH(p_1,p_2):=\cbr{\int\del{\sqrt{p_1(x)}-\sqrt{p_2(x)}}^2\lambda(\d x)}^{1/2}.$

\begin{proof}[Proof of \cref{thm:general}]
For the proof of \ref{thm:general:a}, it suffices to derive the same lower bound of the prior concentration on the KL neighborhood of $\nu^\star*F$ as the one in \cref{eq:prior_mass_lbound}. Let $\tilde\zeta_n:=\sqrt{\log n/n}$. 
By \labelcref{eq:ker:bound} in Assumption \ref{Ker1}, %Theorem 5 of \citeS{wong1995probability}
Lemma 7 of \citeS{ghosal2007posterior} implies that
    \begin{equation*}
        \cB_{\kl}(\tilde\zeta_n, \nu^\star*\Phi, \cM(\Theta))\supset
        \cbr{\nu\in\cM(\Theta):\fH^2(p_{\nu*F},p_{\nu^\star*F})\le c_1\frac{1}{n}}
    \end{equation*}
for some constant $c_1>0$. Moreover, by \labelcref{eq:ker:lip} in Assumption \ref{Ker1}, $\fH^2(f(\cdot, \theta_1), f(\cdot, \theta_2))\le \|f(\cdot, \theta_1)-f(\cdot, \theta_2)\|_1\le c_2|\theta_1-\theta_2|$ for any $\theta_1,\theta_2\in\Theta$ for some constant $c_2>0$.
 Then invoking Lemma 1 of \citeS{nguyen2013convergence}, we have  $\fH^2(p_{\nu_1*F},p_{\nu_2*F})\le c_2 \sW_1(\nu_1, \nu_2)$ for any $\nu_1, \nu_2\in\cM(\Theta)$. Therefore, by Lemma 3 of \citeS{gao2016posterior}, we obtain
    \begin{equations}
    \label{eq:prior_concen_gen}
       \Pi&\del[1]{\cB_{\kl}(\tilde\zeta_n, \nu^\star*F, \cM(\Theta))}\\
        &\ge \Pi\del{\sW_1(\nu,\nu^\star)\le \frac{c_1}{c_2}\frac{1}{n}|k=k^\star}\Pi(k=k^\star)\\
        &\ge\Pi\del{\sum_{j=1}^{k^\star}|w_j-w_j^\star|\le  \frac{c_1}{2c_2\textsf{diam}(\Theta)}\frac{1}{n}|k=k^\star}\\
        &\qquad \times \Pi\del{|\theta_j-\theta_j^\star|\le  \frac{c_1}{2c_2\textsf{diam}(\Theta)}\frac{1}{n}, \forall j\in[k^\star]|k=k^\star}\Pi(k=k^\star)\\
        &\gtrsim \e^{-c_3k^\star\log n}\Pi(k=k^\star)
    \end{equations}
for some constant $c_3>0$, where the last inequality follows from \cref{assume:prior} and the assumption $\textsf{diam}(\Theta)<\infty$.

The only part of the proofs of the rest of results in \cref{sec:main} which depends on properties of Gaussian is to prove \cref{lem:moment_var}, so the results in  \ref{thm:general:b}-\labelcref{thm:general:e}  can be derived  by the same proof techniques under Assumption \ref{Ker2}.
\end{proof}

\begin{proof}[Proof of \cref{thm:general_high}]
In spirit of \cref{lem:moment_high}, it suffices to show posterior contraction of the vector of the first $\hat{k}$ moments, where $\hat{k}\asymp \log n/\log\log n$. This can be proved by the similar argument as the one used in the proof of \cref{thm:general}. 
\end{proof}

\subsection{Proofs for \cref{sec:dp}}

Before providing the proofs, we first introduce the lemma that states a lower bound of the concentration of the Dirichlet distribution. The lemma has been frequently used in the related literature, e.g., Lemma 6.1 of \citeS{ghosal2001entropies} and Lemma 10 of \citeS{ghosal2007posterior}.  But here, we state the lower bound in terms of the Dirichlet parameter $(\kappa_1,\dots, \kappa_k)$ in order to be used in the proof of \cref{thm:k_upper_dp}, which is not the case of the existing statements. The proof is exactly the same, so we omit it.

\begin{lemma}
\label{lem:dirichlet}
Let $(w_1,\dots, w_k)$ be distributed according to the Dirichlet distribution with parameter $(\kappa_1,\dots, \kappa_k)$ such that $\kappa_j\in(0,1]$ for any $j\in[k]$. Then for any $(w_1^0,\dots, w_k^0)\in\Delta_k$ and any $\eta\in(0,1/k]$,
    \begin{equation*}
        \P\del{\sum_{j=1}^k|w_j-w_j^0|\le2\eta}\ge \eta^{2(k-1)}\prod_{j=1}^k\kappa_j.
    \end{equation*}
\end{lemma}

We are ready to prove \cref{thm:k_upper_dp,thm:mixing_dp}.

\begin{proof}[Proof of \cref{thm:k_upper_dp}]
If $k^\star>n$, the event of interest is empty, so we focus on the cases that $k^\star\le n$.
Let $\tilde\zeta_n:=\sqrt{\log n/n}$. As in the proof of \cref{thm:k_upper}, we have that 
    \begin{align*}
        \int\frac{p_{\tilde\nu*\Phi}^{(n)}}{p_{\nu^\star*\Phi}^{(n)}}(X_{1:n}) \Pi_{\textsc{dp}}(\d\tilde\nu)
        \ge \e^{-2n\tilde\zeta_n^2}\Pi_{\textsc{dp}}\del{\cB_{\kl}(\tilde\zeta_n, \nu^\star*\Phi, \cM_\infty)}
    \end{align*}
with $\P_{\nu^\star*\Phi}^{(n)}$-probability at least $1-1/(n\tilde{\zeta}_n^2)$.
%Let $ \xi_n:=(n\log n)^{-1}$ so that $\xi_n \log(1/ \xi_n)\lesssim \xi_n \log^2(1/ \xi_n)\lesssim  \tilde\zeta_n^2$. 
Since $\|p_{\nu^\star*\Phi}/p_{\nu^\star*\Phi}\|_\infty\le \e^{L^2/2}$,
%Then since $\int p_{\nu^\star*\Phi}(x)(p_{\nu^\star*\Phi}(x)/p_{\nu*\Phi}(x))^b \d\lambda(x)<\infty$ for some $b\in(0,1)$  by Equation (4.6) of \citepS{ghosal2001entropies}, Theorem 5 of \citeS{wong1995probability} implies that
Lemma 8 of \citeS{ghosal2007posterior} implies that
    \begin{equation*}
        \Pi_{\textsc{dp}}\del{\cB_{\kl}(\zeta_n, \nu^\star*\Phi, \cM_\infty)}\ge \Pi_{\textsc{dp}}\del{\tilde\nu\in\cM_\infty:\|p_{\tilde\nu*\Phi}-p_{\nu^\star*\Phi}\|_1\le c_1\log n/n}.
    \end{equation*}
for some constant $c_1>0$. Let $B_0,B_1,\dots, B_{k^\star}$ be a partition of $[-L,L]$ such that $\theta_j^\star\in B_j$, $\text{diam}(B_j)=c_1\log n/(4n)$ for each $j\in[k^\star]$ (Here we assume without loss of generality that all the atoms of $\nu^\star$ does not overlap with each other, otherwise, we can consider a partition where each set contains exactly one distinct atom). Since  the vector $(\tilde\nu(B_0), \tilde\nu(B_1), \dots, \tilde\nu(B_{k^\star}))$ follows the Dirichlet distribution with parameter $(\kappa_n H(B_0),\kappa_n H(B_1), \dots, \kappa_n H(B_{k^\star}))$,
by Lemma 5 of \citeS{ghosal2007posterior} and  $\text{diam}(B_j)=c_1\log n/(4n)$ for every $j\in[k^\star],$ we have
    \begin{equation*}
        \cbr{\tilde\nu\in\cM_\infty:\|p_{\tilde\nu*\Phi}-p_{\tilde\nu^\star*\Phi}\|_1\le c_1\log n/n}
        \supset
        \cbr{\tilde\nu\in\cM_\infty:\sum_{j=0}^{k^\star}\abs{\tilde\nu(B_j)-w_j^\star}\le \frac{c_1\log n}{2n}}
    \end{equation*}
with $w_j^\star:=0$. Finally, by \cref{lem:dirichlet},
    \begin{align*}
        \Pi_{\textsc{dp}}\del{\tilde\nu\in\cM_\infty:\sum_{j=0}^{k^\star}\abs{\tilde\nu(B_j)-w_j^\star}\le \frac{c_1\log n}{2n}}
        &\ge \del{\frac{c_1\log n}{4n}}^{2k^\star}\kappa_n^{k^\star+1}\prod_{j=0}^kH(B_j)
        \\&\gtrsim \kappa_n^{k^\star+1}\exp(-c_2k^\star\log n)
    \end{align*}
for some constant $c_2>0$, where the second inequality follows from that $H(B_0)= 1-\sum_{j=1}^{k^\star} H(B_j)\gtrsim 1-1/\log n\gtrsim1$ and $H(B_j)=c_1\log n/(8Ln)$ for $j\in[k^\star]$.% and $\log(1/\xi_n)\asymp  \log n$.

On the other hand, we use Fubini's theorem to obtain
    \begin{align*}
        &\P_{\nu^\star*\Phi}^{(n)}\sbr{\int\sum_{Z_{1:n}\in\bN^n:T_n(Z_{1:n})> Ck^\star}
        \cbr{\prod_{i=1}^n\frac{\phi(X_i- \theta_{Z_i}) p_{w[\tilde{\nu}]}(Z_i) }{p_{\nu^\star*\Phi}(X_{i})}} \Pi_{\textsc{dp}}(\d\tilde\nu)}\\
        & = \int\sum_{Z_{1:n}\in\bN^n:T_n(Z_{1:n})> Ck^\star}\cbr{\int\prod_{i=1}^n \phi(X_i- \theta_{Z_i}) \d X_{1:n}} 
        p_{w[\tilde{\nu}]}^{(n)}(Z_{1:n})\Pi_{\textsc{dp}}(\d\tilde\nu)\\
        & = \sum_{Z_{1:n}\in\bN^n:T_n(Z_{1:n})> Ck^\star}\int p_{w[\tilde{\nu}]}^{(n)}(Z_{1:n})\Pi_{\textsc{dp}}(\d\tilde\nu)\\
        &=\P_{\textsf{CRP}(\kappa_n)}(T_n(Z_{1:n})> Ck^\star),
    \end{align*}
where  $\textsf{CRP}(\kappa_n)$ denotes the Chinese restaurant process with concentration parameter $\kappa_n$.
It is known that the probability mass function of $T_n$ is given by (e.g., see Proposition 4.9 of \citeS{ghosal2017fundamentals})
    \begin{equation*}
        \P_{\textsf{CRP}(\kappa_n)}(T_n=t)
        =C_n(t)n!\kappa_n^t\frac{\Gamma(\kappa_n)}{\Gamma(\kappa_n+n)}
    \end{equation*}
where $C_n(t):=(n!)^{-1}\sum_{S\subset[n-1]:|S|=n-t}\prod_{i\in S}i$.
Since 
    \begin{align*}
        \frac{C_n(t+1)}{C_n(t)}
        &=        \frac{\sum_{S\subset[n-1]:|S|=t}\frac{1}{\prod_{i\in S}i}}{\sum_{S\subset[n-1]:|S|=t-1}\frac{1}{\prod_{i\in S}i}}\\
        &\le \frac{\sum_{S\subset[n-1]:|S|=t-1}\frac{1}{\prod_{i\in S}i}\del{\sum_{i'=1}^{n-1}\frac{1}{i'}}}{\sum_{S\subset[n-1]:|S|=t-1}\frac{1}{\prod_{i\in S}i}}\le \log(\e(n-1))
    \end{align*}
we have
    \begin{equation*}
        \P_{\textsf{CRP}(\kappa_n)}(T_n\ge t+1)
        \lesssim \sum_{h=t+1}^{n}(\kappa_n\log n)^{h-1}\lesssim (\kappa_n\log n)^{t}.
    \end{equation*}
Hence,
    \begin{align*}
        \P_{\nu^\star*\Phi}^{(n)}\sbr{\Pi_{\textsc{dp}}(T_n> C k^*|X_{1:n})}
        &\lesssim \e^{c_3k^\star \log n}\frac{(\kappa_n\log n)^{C k^\star-1}}{\kappa_n^{k^\star+1}}+ o(1) \\
        &\le \e^{c_3k^\star \log n+((C-1)k^\star-2)\log \kappa_n +(C k^\star-1)\log\log n} + o(1)\\
        &\le \e^{ (2a_0-(a_0(C-1)-c_3)k^\star)\log n +(C k^\star-1)\log\log n} + o(1)
    \end{align*}
for some constant $c_3>0$. If $C>c_3/a_0+3$, the desired result follows since $\log\log n/\log n\to0$.   
\end{proof}

\begin{proof}[Proof of \cref{thm:mixing_dp}]
%Let $\tilde\xi_n:=n^{-1/3}\log^{1/2}n$ and $\xi_n:=n^{-2/3}\log ^{-2}n$ so that $\xi_n\log ^2(1/\xi_n)\lesssim \tilde\xi_n^2$. 

Let $\xi_n:=n^{-1/3}$. By the same arguments used in the proof of \cref{thm:mixing_high}, we have that
    \begin{align*}
        D_n&:=\int\frac{p_{\tilde\nu*\Phi}^{(n)}}{p_{\nu^\star*\Phi}^{(n)}}(X_{1:n}) \Pi_{\textsc{dp}}(\d\tilde\nu)\\
        &\ge \e^{-2n\xi_n^2}\Pi_{\textsc{dp}}\del{\tilde\nu \in\cM_\infty:\sum_{j=1}^R|\tilde\nu(B_j)-\nu^\star(B_j)|\le \frac{c_1\xi_n^2}{16L^2}}
    \end{align*}
with $\P_{\nu^\star*\Phi}^{(n)}$-probability at least $1- 1/n\xi_n^2$, where we define $R:=\ceil{4L/\sqrt{c_1\xi_n^2}}$ and $(B_1, \dots, B_R)$ is a partition of $[-L,L]$ such that $\textsf{diam}(B_j)\le \sqrt{c_1\xi_n^2}/2$ for some $c_1>0$. Since $(\tilde\nu(B_1), \dots, \tilde\nu(B_R))$ follows the Dirichlet distribution with parameter $(\kappa_n H(B_1),\dots, \kappa_n H(B_R))$, by \cref{lem:dirichlet}, $D_n$ is further bounded as
    \begin{align*}
        D_n\gtrsim \e^{-2n\xi_n^2} (\xi_n^2)^{2R}(\kappa_n\xi_n)^{R}\gtrsim \e^{-c_2n^{1/3}\log^{b_0} n}
    \end{align*}
with $\P_{\nu^\star*\Phi}^{(n)}$-probability at least $1- 1/n\xi_n^2$. Following the proof of \cref{thm:mixing_high}, we obtain the desired result.
\end{proof}

%\subsection{Proofs of lemmas in \cref{sec:proof:mixing}}
%\label{sec:proof:lemmas}

%In this subsection, we provide the proofs of the two lemmas given in  \cref{sec:proof:mixing}, which were used to prove \cref{thm:mixing}.

\section{Analysis of general mixture models in the framework of \citet{heinrich2018strong}}
\label{sec:heinrich}

In this section, we analyze general mixture models satisfying strong identifiability conditions based on the theoretical tools provided in  \citet{heinrich2018strong}. Although the Gaussian mixture model considered in \cref{sec:main} is a special case of the  mixture models considered in this section, we cannot immediately derive the results of \cref{sec:main} since  the framework of \citet{heinrich2018strong} assumes the fixed number of components, while the analysis of \cref{sec:main} allows the growing number of components.

We assume here that the data are i.i.d observations from the distribution $\nu^\star\bullet F$ for some $k^\star$-atomic mixing distribution $\nu^\star\in\cM_{k^\star}$ and family of distribution functions $\{F(\cdot, \theta):\theta\in\Theta\}$ satisfying some regularity and strong identifiability conditions. We first introduce the strong identifiability condition, which appeared as Definition 2.2. in \citeS{heinrich2018strong}.

\begin{definition}
\label{def:strong}
A family of distribution functions $\cbr{F(\cdot, \theta):\theta\in\Theta}$ for $\Theta\subset \bR$, is said to be \textit{$q$-strongly identifiable} if for any finite subset $B$ of $\Theta$, 
    \begin{equation*}
        \sup_{\x\in\R}\abs{\sum_{j=0}^q\sum_{\theta'\in B}a_{j, \theta'}\frac{\partial^j f}{\partial\theta^j}(\x, \theta')}=0
        \Longrightarrow \max_{j\in\{0,1,\dots, q\}}\max_{\theta'\in B}|a_{j,\theta'}|=0.
    \end{equation*}
We say that a mixture distribution $\nu*F$ is $q$-strongly identifiable if $\cbr{F(\cdot, \theta):\theta\in\Theta}$ is $q$-strongly identifiable.
\end{definition}

\citet[][Theorem 2.4]{heinrich2018strong} shows that the location mixture models, i.e., $f(x,\theta)=f(x-\theta)$, in which both the kernel density function $f(\cdot)$ and its derivatives up to $q-1$-th order vanish at $\pm\infty$, are $q$-strongly identifiable. Thus the Gaussian location mixture model we consider is $\infty$-strongly identifiable. Also the scale mixtures, i.e., $f(x,\theta)=\theta^{-1}f(\theta^{-1}x)$ for $\theta\in\Theta\subset\R_+$, with the same condition on the kernel density function, are $q$-strongly identifiable.

We impose the following regularity conditions including the strong identifiability condition.

\begin{assumption*}{F$^*$(\textit{q})}
\label{assume:cdf_q}
The family $\cbr{F(\cdot, \theta):\theta\in\Theta}$ of  distribution functions on $\R$ satisfies Assumptions \labelcref{Ker0} and \labelcref{Ker1} as well as the following conditions:
\begin{enumerate}[label=(F$^*$\arabic*)]
    \item \label{f_a1} For any $x\in\bR$, $F(x,\theta)$ is $q$-differentiable with respect to $\theta$.
        
    \item \label{f_a2} $\cbr{F(\cdot, \theta):\theta\in\Theta}$ is $q$-strongly identifiable.

\end{enumerate}
\end{assumption*}

The additional conditions \labelcref{f_a1} and \labelcref{f_a2} are  inherited from the regularity condition of \citepS{heinrich2018strong}. 

In this section, we assume that the number of components $k^\star$ is fixed but still unknown. We thus use the prior distribution on the number of components satisfying \ref{p_a1} with the constant $\bar{k}_n$. Note that we still allow the true mixing distribution to vary with the sample size. This setup is still substantially more general than  the fixed truth setup considered in the existing Bayesian literature \citepS{nguyen2013convergence, scricciolo2017bayesian, guha2019posterior}.

The following theorem shows the posterior contraction rate for the strongly identifiable mixtures under this setup.

\begin{theorem}
\label{thm:mixing_strong}
Assume that the family $\{F(\cdot, \theta):\theta\in\Theta\}$ of distribution functions on $\R$ satisfies \cref{assume:cdf_q} with $q=2k^\star$. Then with the prior distribution $\Pi$ satisfying  \cref{assume:prior}, we have
    \begin{equation}
    \label{eq:mixing_conv_strong}
         \sup_{\nu^\star\in\cM_{k^\star}(\Theta)}\P_{\nu^\star\bullet F}^{(n)}\sbr{\Pi_F\del{\sW_1(\nu,\nu^\star)\ge M\del{\frac{\log n}{ n}}^{\frac{1}{4k^\star-2}}\big|X_{1:n}}}=o(1)
    \end{equation}
for some  constant $M>0$ depending only on $k^\star$.
\end{theorem}

\begin{proof}%[Proof of \cref{thm:mixing_strong}]
We introduce the notation
    \begin{equation*}
        F(\cdot, \nu):=\int F(\cdot, \theta )\nu(\d\theta)
    \end{equation*}
for a distribution function $F(\cdot, \cdot)$ and the mixing distribution $\nu\in\cM$, which denotes the distribution function of $\nu*F.$  Let $\tilde\zeta_n:=\sqrt{\log n/n}$. By (19) in Theorem 6.3 of \citeS{heinrich2018strong}, we have that
    \begin{align}
        &\cbr{\nu\in\cM(\Theta):\sW_1(\nu, \nu^\star)\ge  M\del{\frac{\log n}{ n}}^{\frac{1}{4k^\star-2}} }\nonumber\\
        &\subset  \cbr{\nu\in\cM_{k^\star}(\Theta):\|F(\cdot, \nu)-F(\cdot, \nu^\star)\|_\infty\ge  c_1M\del{\frac{\log n}{ n}}^{1/2} }\cup\cbr{\nu\notin\cM_{k^\star}(\Theta)}
        \label{eq:decompose_strong}
    \end{align}
for some constant $c_1>0$. By the similar argument of \cref{thm:k_upper}, it is not hard to prove that the expected posterior probability of the event $\cbr{\nu\notin\cM_{k^\star}}$ goes to zero. On the other hand, combining \labelcref{eq:prior_concen_gen} in the proof of \cref{thm:general} and Lemma 1 of \citeS{scricciolo2017bayesian}, the expected posterior probability of the first event in  \labelcref{eq:decompose_strong} also goes to zero, which completes the proof.
\end{proof}

%The convergence rate in \cref{thm:mixing_strong} is equivalent to the convergence rate \labelcref{eq:rate} for the Gaussian mixtures with the fixed number of components $k^\star$.

\begin{remark}
As we mentioned before, although the Gaussian mixture model considered in \cref{sec:main} satisfies \cref{assume:cdf} with $q=\infty$, we cannot immediately derive \cref{thm:mixing} from \cref{thm:mixing_strong} since the latter theorem assumes the fixed number of components.  We believe, however, that even if the number of components grows, the result of \cref{thm:mixing_strong} still holds with the same convergence rate as  \labelcref{eq:mixing_conv_strong} up to a constant depending on $k^\star$, provided that  \cref{assume:cdf} is met with $q=\infty$.  We need to establish a uniform version of Theorem 6.3 of \citeS{heinrich2018strong} over the number of components, which is a key technical tool for the proof. It could be an objective of future work.
\end{remark}

Moreover, our Bayesian procedure can obtain the minimax optimal convergence rate  \citepS[][Theorem 3.2]{heinrich2018strong} under the locally varying condition on the true mixing distribution, which is assumed in \cref{thm:mixing_local} for the Gaussian mixtures.

\begin{theorem}
\label{thm:mixing_strong_local}
Let $k^\star,k_0\in\bN$ be fixed constants with $k^\star\ge k_0$ and let $\nu_0\in\cM_{k_0}(\Theta)\setminus\cM_{k_0-1}(\Theta)$ be a fixed distribution. Assume that the family $\{F(\cdot, \theta):\theta\in\Theta\}$ of distribution functions on $\R$  \cref{assume:cdf_q} with $q=2k^\star$. Moreover, assume that the prior distribution $\Pi$ satisfies  \cref{assume:prior}. Then there exist  constants $\tau>0$ and $M>0$ depending only on $k^\star$ and $k_0$ such that 
    \begin{equation}
        \sup_{\nu^\star\in\cM_{k^\star}(\Theta):\sW_1(\nu^\star, \nu_0)<\tau}\P_{\nu^\star\bullet F}^{(n)}\sbr{\Pi_F\del{\sW_1(\nu,\nu^\star)\ge M\del{\frac{\log n}{ n}}^{\frac{1}{4(k^\star-k_0)+2}}\big|X_{1:n}}}=o(1).
    \end{equation}
%for any $\nu^\star\in\cM_{k^\star}$ with  $\sW_1(\nu^\star, \nu_0)<\tau$ eventually.
\end{theorem}

\begin{proof}%[Proof of \cref{thm:mixing_strong_local}]
Using the similar argument in the proof of \cref{thm:mixing_strong} combined with  (18) in Theorem 6.3 of \citeS{heinrich2018strong}, we obtain the desired result.
\end{proof}

\bibliographystyleS{imsart-nameyear}
\bibliographyS{_references}
%\subsection{Proofs} % for \cref{sec:general}
%\label{appen:prf_general}

%\input{mixtures_general.tex}
\end{document}